\documentclass[11pt, twoside]{article}
\usepackage{latexsym}
\usepackage{amsmath}
\usepackage{amssymb}
\usepackage[all]{xy}
\usepackage{amsfonts}
\usepackage{verbatim}
\usepackage{amsthm}
\usepackage{indentfirst}
\usepackage{mathrsfs}\usepackage{array}
\usepackage{epsfig}
\usepackage{xy}
\usepackage{tensor}\usepackage{tocbibind}
\usepackage{array}
\usepackage{stmaryrd}
\usepackage{graphicx,color}
\usepackage{xcolor}
\usepackage[colorlinks=true,linkcolor=blue,citecolor=blue]{hyperref}
\usepackage{tikz}
\usetikzlibrary{arrows,calc}
\usepackage{etex}
\usepackage{mathdots}
\usepackage{float}
\usepackage{graphics}
\usepackage{pdflscape}
\usepackage{extarrows}
\usepackage{anysize,hyperref}
\input xypic
\xyoption{all}
\usepackage{bm}
\usepackage[perpage,symbol]{footmisc}
\usepackage{setspace}
\SelectTips{cm}{10}
\newcommand{\twoheaddownarrow}{\mathrel{\rotatebox[origin=c]{-90}{$\twoheadrightarrow$}}}

\def\dr{\ar@{->}[r]}

\begin{document}
\baselineskip=15pt
\title{\Large{\bf Heaps of modules: homological aspects
\footnotetext{$^\ast$Corresponding author.  ~Yongduo Wang is supported by the National Natural Science Foundation of China (Grant No. 12661005). Jian He is supported by the National Natural Science Foundation of China (Grant Nos. 12501048, 12661005) and the Hongliu Outstanding Young Talents Funding of Lanzhou University of Technology. ~Dejun Wu is supported by the National Natural Science Foundation of China (Grant Nos. 12261056, 12661005).} }}
\medskip
\author{Yongduo Wang,~Chenyu Wang,~Jian He$^\ast$ and Dejun Wu}
\date{}
\maketitle
\def\blue{\color{blue}}
\def\red{\color{red}}

\newtheorem{theorem}{Theorem}[section]
\newtheorem{lemma}[theorem]{Lemma}
\newtheorem{corollary}[theorem]{Corollary}
\newtheorem{proposition}[theorem]{Proposition}
\newtheorem{conjecture}{Conjecture}
\theoremstyle{definition}
\newtheorem{definition}[theorem]{Definition}
\newtheorem{question}[theorem]{Question}
\newtheorem{notation}[theorem]{Notation}
\newtheorem{remark}[theorem]{Remark}
\newtheorem{remark*}[]{Remark}
\newtheorem{example}[theorem]{Example}
\newtheorem{example*}[]{Example}

\newtheorem{construction}[theorem]{Construction}
\newtheorem{construction*}[]{Construction}

\newtheorem{assumption}[theorem]{Assumption}
\newtheorem{assumption*}[]{Assumption}

\baselineskip=17pt
\parindent=0.5cm

\begin{abstract}
\noindent \\[0.3cm]
Homological aspects of the theory of  heaps of modules are studied in this paper.
The definitions of projective objects and Gorenstein projective objects in the category of heaps of $T$-modules are posed, where $T$ is a truss. It is shown that a heap of $T$-modules $P$ is projective if and only if $\mathcal{G}_{e_p}(P)$ is a projective $R(T)$-module for all $e_p\in P$ and a heap of $T$-modules $M$ is BP Gorenstein projective if and only if $\mathcal{G}_{e_m}(M)$ is a Gorenstein projective $R(T)$-module for all $e_m\in M$. Moreover, we give a functorial description of the BP Gorenstein projective dimension. Finally, it is also proven that a unital truss $T$ is a PA Gorenstein truss if and only if $R(T)$ is an Iwanaga-Gorenstein ring.

\medskip
\end{abstract}
{\footnotesize\tableofcontents}
\medskip

\pagestyle{myheadings}
\markboth{\rightline {\scriptsize }}
         {\leftline{\scriptsize  }}

 \medskip

\pagestyle{myheadings}
\markboth{\rightline {\scriptsize Y. D. Wang, C. Y. Wang, J. He and D. J. Wu\hspace{2mm}}}
         {\leftline{\scriptsize   Heaps of modules: homological aspects}}

\section{Introduction}
The notion of heaps was introduced by H. Pr$\ddot{u}$fer \cite{hp}, R. Baer \cite{tb4} and A. K. Su$\check{s}$kevi$\check{c}$ \cite{aks} in the 1920s. A heap is a set $H$ together with a ternary operation
$[---]: H\times H\times H\rightarrow H$ which is associative and satisfies the Mal'cev identities, that is,
\begin{equation*}
[[a,b,c],d,e] = [a,b,[c,d,e]] \qquad \text{and} \qquad [a,b,b] = a = [b,b,a]
\end{equation*}
for all $a,b,c,d,e\in H$. It exhibits that there is a deep connection between groups and heaps.

In 2019, trusses were introduced by T. Brzezi$\acute{n}$ski in \cite{tb} as structures describing two different distributive laws: the well-known ring distributivity and the one coming from the recently introduced braces, which are gaining popularity due to their roles in the study of the set-theoretic solutions of the Yang-Baxter equation. The brace distributive law appeared earlier in the context of quasi-rings of radical rings (see \cite{ch}). It turns out that rings and braces can be described elegantly by switching the group structure to a heap structure. This leads to the definition of a truss, which is a set $T$ with a ternary operation $[-, -, -]$ and a binary multiplication $\cdot$ satisfying some conditions, the crucial one being the generalisation of ring and brace distributivity: $a\cdot [b, c, d] = [a\cdot b, a\cdot c, a\cdot d]$ and $[b, c, d]\cdot a = [b\cdot a, c\cdot a, d\cdot a]$, for all $a, b, c, d\in T$. Due to this, we can jointly approach brace and ring theory.

A truss can be understood as a ring in which the abelian group of addition has no specified neutral element. Also, every truss $T$ is a congruence class of a ring $R(T)$, the universal extension of $T$ into a ring (see \cite{tb1}). Trusses, even though close to rings, differ significantly as the category of trusses has no zero object. It is well known that, for a ring $R$, we must study modules over it, so it is natural to ask: what is the theory of modules over trusses? The notion of modules over trusses was posed and basic properties of it were given by T. Brzezi$\acute{n}$ski (see \cite{tb}). In recent years, modules over trusses were studied by S. Breaz, T. Brzezi$\acute{n}$ski, B. Rybolowicz, P. Saracco and other relevant researchers from different aspects (see \cite{stbp,mb,rtb,tb2,tbp,tb5,tb6}).

In order to study the affine version of the module $M$ over a truss $T$, the definition of heap of $T$-modules was posed by S. Breaz, T. Brzezi$\acute{n}$ski, B. Rybolowicz and P. Saracco in \cite{stbp,mb}. A heap of $T$-modules is an abelian heap $M$ together with a ternary operation $\rhd:T \times M \times M \rightarrow M$ such that, for each $m\in M$, the map $(t,n)\mapsto t\rhd_{m}n$ defines a left $T$-module structure and the base change property holds. The category $T\textbf{-HMod}$ of heaps of $T$-modules consists of such objects with heap homomorphisms that preserve $\rhd$. It was shown that heaps of modules over a truss
offer a promising algebraic setting for a frame-independent approach to mathematical notions that come
from various domains. For example, the heaps of $T$-modules approach to affine spaces allows one for a smooth formulation of Lie brackets on affine spaces or Lie affgebras (see
\cite{tb2,ggu}). In addition, particular solutions of the set theoretic Yang-Baxter equation coming from spindles and quandles can be parameterized by using heaps of modules. So the research of heap of $T$-modules is very meaningful.

It is well known that projective objects and Gorenstein projective objects play important roles in category theory and homological algebra. Here we will study projective objects and Gorenstein projective objects in the category of heaps of $T$-modules. The paper is arranged as follows. In Sectin 2, we provide some definitions and basic properties of heaps, trusses and modules over trusses that will be used in this work. In Section 3, the definition of projective objects in the category $T\textbf{-HMod}$ is introduced. And we prove the following results. It is worth nothing that these proofs differ drastically from those in classical module theory.

\begin{theorem}\label{main1} { (see Theorems \ref{th1} and \ref{th2} for details)  Let $P$ be a heap of T-modules. Then the following statements are equivalent.
}
\begin{itemize}
\item[ (1)] $P$ is projective;
\item[(2)] $\mathcal{G}_{e_p}(P)$ is a projective $R(T)$-module, for all $e_p\in P $;
\item[(2)] $P$ is a coproduct summand of a free heap of $T$-modules.
\end{itemize}
\end{theorem}

\begin{theorem}\label{main1} { (see Proposition \ref{th3} for details)
Let $P$ be a heap of $T$-modules. The following conditions are equivalent.}
\begin{enumerate}
    \item[(1)] $P$ is projective;
    \item[(2)] The functor $\operatorname{Hom}_{TH}(P,-): T\textbf{-HMod} \rightarrow  \textbf{Ah}$ preserves epimorphisms;
    \item[(3)] Every short exact sequence $K \stackrel{\tau}{\hookrightarrow} N \stackrel{\pi}{\twoheadrightarrow} P$ splits.
\end{enumerate}
\end{theorem}
The Gorenstein homological dimensions in the category of modules over a non-trivial associative ring are fully developed in \cite{rb}. However, $T\textbf{-HMod}$ is not even an abelian category, and the concept of homology does not exist in it. Thus the methods developed in the foregoing-mentioned paper can no longer be applied directly. To circumvent this obstacle here, we introduce some concepts, such as BasePoint-Compatible sequences, TH-Complex etc. in Section 4. Then, by means of the functors $\mathcal{G}$ and $\mathcal{H}$, we relate these concepts to the corresponding concepts in the category of modules over a non-trivial associative ring. The definition of Gorenstein projective objects in the Category $T\textbf{-HMod}$ is posed. It is shown that a heap of $T$-modules $M$ is BP Gorenstein projective if and only if $\mathcal{G}_{e_m}(M)$ is a Gorenstein projective $R(T)$-module for all $e_m\in M$. And we also prove the following result.
\begin{theorem}{(see Theorem \ref{th4} for details)
Let $M$ be a heap of $T$-modules with finite BP Gorenstein projective dimension and $n$ an integer. Then the following conditions are equivalent.}
\begin{enumerate}
    \item[(i)] $BPGpd_{TH}(M)\leq n $;
    \item[(ii)] ${\rm Ext}_{TH}^{i}(M, L) = e_i$ for all $i > n$, and for all heaps of $T$-modules $L$ with finite $BPpd_{TH}(L)$;
    \item[(iii)] ${\rm Ext}_{TH}^{i}(M, Q) = e_i$ for all $i > n$, and for all projective heaps of $T$-modules $Q$;
    \item[(iv)] For every BasePoint-Compatible sequence of heaps of $T$-modules $K_{n} \hookrightarrow G_{n-1} \to \cdots \to G_{0} \twoheadrightarrow M$, where $G_{0}, \dots, G_{n-1}$ are BPGorenstein projective, $K_{n}$ is BP Gorenstein projective.
\end{enumerate}
\end{theorem}
In Section 5, we introduce the definition of (PA) Gorenstein trusses. It is proven that a unital truss $T$ is a PA Gorenstein truss if and only if $R(T)$ is an Iwanaga-Gorenstein ring; see Theorem \ref{rt}. In particular, we finally provide an equivalent characterization of Iwanaga-Gorenstein unital rings; see Corollary \ref{sj}.

These works are of great significance for further research on
the (Gorenstein) homological theory of $T\textbf{-HMod}$.

\section{Preliminaries}
In this section, we provide some definitions and basic properties of heaps, trusses and modules over trusses that will be used in this work. We omit some details here, but the reader can find them in \cite{mb,tb,rtb,tb2, tbp}.
\subsection{Heaps}
A heap is a set $H$ together with a ternary operation
$[-,-,-]: H\times H\times H\longrightarrow H$ which is associative and satisfies the Mal'cev identities, that is,
\begin{equation*}
[[a,b,c],d,e] = [a,b,[c,d,e]] \qquad \text{and} \qquad [a,b,b] = a = [b,b,a]
\end{equation*}
for all $a,b,c,d,e\in H$.
A heap $H$ is said to be abelian if for all $a,b,c\in H$, $[a,b,c] = [c,b,a]$.

A heap morphism from $(H,[-,-,-])$ to $(H^\prime,[-,-,-])$ is a function $f:H\longrightarrow H^\prime$ respecting the ternary operations, i.e., such that for all $x$, $y$, $z\in H$, $f([x,y,z])=[f(x),f(y),f(z)]$.
The category of heaps is denoted by \textbf{Heap} and the category of abelian heaps is denoted by \textbf{Ah}. A singleton set $\{\ast\}$  with the unique heap operation $[\ast,\ast,\ast]=\ast$, it is the terminal object in the category of heaps. The empty set is the initial object. There is no zero object in the category of heaps.

\begin{lemma}\label{le.se}{\rm \cite[Lemma 2.3]{rtb}}
Let $(H,[-,-,-])$ be a heap.
\begin{enumerate}
\item[{\rm (1)}] If $e$, $x$, $y\in H$ are such that $[x,y,e]=e$ or $[e,x,y]=e$, then $x=y$.
\item[{\rm (2)}] For all $v$, $w$, $x$, $y$, $z\in H$,
$$[v,w,[x,y,z]]=[v,[y,x,w],z].$$
\item[{\rm (3)}] For all $x$, $y$, $z\in H$,
$$[x,y,[y,x,z]]=[[z,x,y],y,x]=[x,[y,z,x],y]=z.$$
In particular, in the expression $[x,y,z]=w$, any three elements determine the fourth one.
\item[{\rm (4)}] If $H$ is abelian, then, for all $x_i$, $y_i$, $z_i\in H$, $i=1,2,3$,
$$[[x_1,x_2,x_3],[y_1,y_2,y_3],[z_1,z_2,z_3]]=[[x_1,y_1,z_1],[x_2,y_2,z_2],[x_3,y_3,z_3]].$$
\end{enumerate}
\end{lemma}

A subset $S$ of a heap $H$ that is closed under the heap operation is called a sub-heap of $H$. Every non-empty sub-heap $S$ of an abelian heap $H$ defines a congruence relation $\sim_S$ on $H$:
\begin{equation*}
a\sim_S b \quad \iff \quad \exists~s\in S,\ [a,b,s]\in S \quad \iff \quad \forall~s\in S,\ [a,b,s]\in S.
\end{equation*}
The equivalence classes of $\sim_S$ form an abelian heap with operation induced from that in $H$. Namely,
$
[\bar a, \bar b, \bar c] = \overline{[a,b,c]}
$, where $\bar x$ denotes the class of $x$ in $H/\sim_S$ for all $x\in H$. This is known as the {\em quotient heap} and it is denoted by $H/S$.
For any $s\in S$, the class of $s$ is equal to $S$.
\subsection{Trusses}

\begin{definition}{\rm \cite[Definition 3.1]{rtb}}
A truss is an algebraic system consisting of a set $T$, a ternary operation $[-,-,-]$ making $T$ into an Abelian heap, and an associative binary operation $\cdot$ which distributes over $[-,-,-]$, that is, for all $w$, $x$, $y$, $z\in T$,
\begin{equation*}
w[x,y,z]=[wx,wy,wz], \quad  [x,y,z]w=[xw,yw,zw].
\end{equation*}
A truss is said to be commutative (abelian) if the binary operation $\cdot$ is commutative.
\end{definition}
A heap homomorphism between two trusses is a truss homomorphism if it respects multiplications. The category of trusses and their morphisms is denoted by ${\rm \textbf{Trs}}$.

Let $T$ be a truss. A left $T$-module is an abelian heap $M$ together with an associative left action $\lambda_M: T\times M \longrightarrow M$ of $T$ on $M$ that distributes over the heap operation. The action is denoted on elements by $t\cdot m = \lambda_M(t,m)$, with $t\in T$ and $m\in M$. Explicitly, the axioms of an action state that, for all $t,t',t''\in T$ and $m,m',m''\in M$,
\begin{subequations}
\begin{equation*}
t\cdot(t'\cdot m) = (tt')\cdot m,
\end{equation*}
\begin{equation*}
 [t,t',t'']\cdot m = [t\cdot m,t'\cdot m,t''\cdot m] ,
\end{equation*}
\begin{equation*}
 t\cdot [m,m',m''] = [t\cdot m,t\cdot m',t\cdot m''].
\end{equation*}
\end{subequations}

\begin{definition}{\rm \cite[Definition 2.7]{mb}}
Let $T$ be a truss. A {\it pointed module over $T$}, or {\it pointed $T$-module}, is an abelian group $G$ together with an action \mbox{$\cdot\colon T\times G\to G$} of the multiplicative semigroup of $T$ on $G$ such that for all $t,t',t''\in T$ and $g,h\in G$,
\begin{equation}\label{eq:TGrp}
[t,t',t'']\cdot g = t\cdot g - t'\cdot g + t''\cdot g \qquad  \text{and} \qquad t\cdot (g+h) = t\cdot g+t\cdot h.
\end{equation}
If $T$ is unital and $1_T \cdot g = g$ for all $g \in G$, then $G$ is called a {\em unital} pointed $T$-module.

A morphism of pointed $T$-modules is by definition a group homomorphism $f \colon  G \to G'$ such that
$f(t \cdot g) = t \cdot f(g)$
for all $g \in G$ and $t \in T$. For the sake of brevity, we may often call them \emph{$T$-linear group homomorphisms}.
All pointed $T$-modules together with $T$-linear group homomorphisms form the category $T\textbf{-Mod}_{\bullet}$.
\end{definition}

\subsection{The category $T\textbf{-HMod}$ of heaps of $T$-modules}
\begin{definition}{\rm \cite[Definition 2.5]{mb}}
Let $T$ be a truss. A \emph{heap of $T$-modules} $(M,[-,-,-],\triangleright)$ is an abelian heap $(M,[-,-,-])$ together with an operation
\[\triangleright : T \times M \times M \to M, \qquad (t,m,n) \mapsto t \triangleright_m n.\]
such that
\begin{enumerate}
\item[(HM1)] for all $m \in M$, the operation
\[\triangleright_m : T \times M \to M, \qquad (t,n) \mapsto t \triangleright_m n,\]
makes $M$ a left $T$-module,
\item[(HM2)] the operation $\triangleright$ satisfies the \emph{base change property}
\[t \triangleright_m n = [ t \triangleright_e n, t \triangleright_e m, m],\]
for all $ m,n,e \in M$, $t \in T$.
\end{enumerate}
 For the sake of brevity, we will often denote a heap of $T$-modules simply by $(M,\triangleright)$, leaving the heap structure $[-,-,-]$ understood. Such a heap of $T$-modules $M$ is called a left heap of $T$-modules, denoted by $_{T} M$. Similarly, one may define a right heap of $T$-modules.

A \emph{morphism of heaps of modules} is a morphism $f : M \to N$ of heaps such that
$$ f(t \triangleright_m n) = t \triangleright_{f(m)} f(n). $$
for all $m,n \in M$, $t \in T$. The category of heaps of $T$-modules and their morphisms will be denoted by $T\textbf{-HMod}$.
Suppose that $T$ is unital with unit $1_T$. A heap of $T$-modules $M$ is said to be \emph{isotropic} if $1_T \triangleright_m n = n$ for all $m,n \in M$. The full subcategory of isotropic heaps of $T$-modules is denoted by $T\textbf{-HMod}_{\textbf{is}}$.
\end{definition}

\begin{remark}{\rm \cite[Remark 2.16]{mb}}
Let $T$ be a truss and $G$ a pointed $T$-module. If we consider the map
\[\triangleright : T \times G \times G \longrightarrow G, \qquad (t,x,y) \longmapsto t \triangleright_x y = t \cdot y - t \cdot x + x,\]
then $\mathcal{H}(G) := (\mathrm{H}(G),\triangleright)$ is a non-empty heap of $T$-modules. Furthermore, the assignment $\mathcal{H} : (G,\cdot) \longmapsto (\mathrm{H}(G),\triangleright)$ from the category of pointed $T$-modules to the category of non-empty heaps of $T$-modules is functorial.
In the opposite direction, let $(M,\triangleright)$ be a non-empty heap of $T$-modules. Then for any chosen $e \in M$, we can consider
\[\triangleright_e : T \times M \to M\]
and the pair $(\mathrm{G}(M;e),\triangleright_e)$ turns out to be a pointed $T$-module, which we denote by $\mathcal{G}((M,\triangleright);e)$. The assignment $\mathcal{G} : (M,\triangleright) \longmapsto (\mathrm{G}(M;e),\triangleright_e)$ from the category of non-empty heaps of $T$-modules to the one of pointed $T$-modules is itself functorial, if for every morphism of heaps of $T$-modules $\varphi : (M,\triangleright) \to (N,\triangleright)$ and for every chosen $e \in M$ and $f \in N$, we consider
\[\mathcal{G}(\varphi) := \tau_{f}^{\varphi(e)} \circ \varphi : (\mathrm{G}(M;e),\triangleright_e) \longrightarrow (\mathrm{G}(N;f),\triangleright_f).\]
Moreover, one may observe that $\mathcal{H}\mathcal{G}((M,\triangleright);e) = (M,\triangleright)$ for every non-empty heap of $T$-modules $(M,\triangleright)$ and for all $e \in M$, and $\mathcal{G}(\mathcal{H}(G,\cdot);e) = (\mathrm{G}(\mathrm{H}(G);e),\triangleright_e) \cong (G,\cdot)$ via the translation automorphism $\tau_e^{0_G}$, for every pointed $T$-module $(G,\cdot)$.
\end{remark}

\begin{theorem}{\rm \cite[Theorem 3.1]{mb}}\label{xx}\rm
~The category $T\textbf{-Mod}_{\bullet}$ of pointed $T$-modules is isomorphic to the usual, ring-theoretic, category $\mathrm{R}(T)\textbf{-Mod}$ of $\mathrm{R}(T)$-modules. In particular, it is an abelian category. Furthermore, if $T$ is unital, then the foregoing isomorphism restricts to the categories of unital pointed $T$-modules and unital $\mathrm{R}(T)$-modules as well.
\end{theorem}

\begin{theorem}{\rm \cite[Theorem 3.8]{mb}}\rm
~The category of heaps of $T$-modules is isomorphic to the category of affine $\mathrm{R}(T)$-modules.
\end{theorem}

\begin{construction}{\rm \cite[Construction 4.10]{mb}}\label{xu}
The coproduct of an empty family is the initial object of the category of heaps of $T$-modules (the empty heap). Since the empty heap of $T$-modules is the initial object in our category, it is enough to study the coproduct of a family of non-empty members.
Let $(G_{i})_{i\in I}$ be a family of non-empty heaps of $T$-modules. In order to construct a coproduct for it, we proceed as follows:
\begin{enumerate}
    \item[(a)] for every $i\in I$, we fix a pointed $T$-module structure $(G_{i},+)$ associated to $G_{i}$;
    \item[(b)] we fix an index $i_{0}\in I$ and a family of symbols $(b_{i})_{i\in I\setminus\{i_{0}\}}$;
    \item[(c)] let $C$ be the heap associated with the pointed $T$-module $(\oplus_{i\in I}G_{i})\oplus F$, where $F$ is the free pointed $T$-module which has as a basis the family $(b_{i})_{i\in I\setminus\{i_{0}\}}$;
    \item[(d)] if $u_{i}:G_{i}\to\oplus_{i\in I}G_{i}$ are the canonical morphisms, we define $v_{i_{0}}=u_{i_{0}}$, and $v_{i}=u_{i}+b_{i}$ for all $i\in I\setminus\{i_{0}\}$.
\end{enumerate}
Then the pair $(C,(v_{i})_{i\in I})$ represents a coproduct of the family $(G_{i})_{i\in I}$ in $T$\textbf{-HMod}.
\end{construction}

\begin{corollary}{\rm \cite[Corollary 4.15]{mb}}\label{zzz}\rm
~Let $T$ be a truss, not necessarily unital. Then the category of heaps of $T$-modules is equivalent to the full subcategory $T \textbf{-Mod}_{\bullet}\twoheaddownarrow \mathrm{R}(T)_{u}$ of the slice category of pointed $T$-modules over $\mathrm{R}(T)_{u}$.
\end{corollary}

\begin{definition}{\rm \cite[Definition 5.1]{mb}}
Let
\begin{equation}
M \xrightarrow{f} N\xrightarrow{g} P
\end{equation}
be a sequence of heaps of $T$-modules and morphisms of heaps of $T$-modules. We may suppose, without loss of generality, that $P \neq \varnothing$.
We say that the sequence is \textit{exact in $N$ at $e$} if there exists $e\in g(N)$ such that $f(M)=g^{-1}(e)$.
\end{definition}

\section{Projective objects in the category $T\textbf{-HMod}$}

The definition of projective objects in the category $T\textbf{-HMod}$ is introduced. The  equivalent characterizations of projective objects in the category $T\textbf{-HMod}$ will be systematically studied throughout this section.

\subsection{Projective heaps of $T$-modules}

\begin{definition}
Let $X$ be a set. The free heap of $T$-modules $F_{TH}(X)$ is generated by $X$:
$$F_{TH}(X):=\mathcal{H}(\bigoplus_{x\in X}R(T)_u),$$
where $R(T)_u$ is the Dorroh extension of the ring $R(T)$.
\end{definition}

\begin{definition}\label{proj}
A heap of $T$-modules $P$ is said to be projective if for any epimorphism $f:M\rightarrow N$ in $T\textbf{-HMod}$ and any morphism $g:P\rightarrow N$, there exists a morphism $h:P\rightarrow M$ such that $f\circ h=g$. Diagrammatically,
$$ \xymatrix{ & P \ar@{-->}[dl]_{\exists h} \ar[d]^{g} & \\
M \ar[r]_{f} & N } $$
The class of projective heaps of $T$-modules is denoted by $\mathcal{P}(TH)$.
\end{definition}

\begin{remark}
Theorem \ref{xx} implies that the $T$-module action on a pointed $T$-module induces an $R(T)$-module action, thereby endowing it with the structure of an $R(T)$-module, conversely holds for an $R(T)$-module. Which means a pointed $T$-module can be naturally regarded as an $R(T)$-module, and the converse holds.
\end{remark}
\subsection{Equivalent characterizations of projective heaps of $T$-modules}
\begin{theorem}\label{th1}\rm
A heap of $T$-modules $P$ is projective if and only if $\mathcal{G}_{e_p}(P)$ (fix an element $e_p\in P$, denote by $\mathcal{G}_{e_p}(P)$ the pointed $T$-module obtained from $P$ by retracting at $e_p$) is a projective $R(T)$-module, for all $e_p\in P \Leftrightarrow \exists~ e_p\in P$.
\end{theorem}
\begin{proof}
Necessity: Assume that $P$ is a projective heap of $T$-modules and fix an element $e_p\in P$, $\mathcal{G}_{e_p}(P)$ can be naturally regarded as an $R(T)$-module. Now let $\pi: A \to B$ be an epimorphism in $T\textbf{-Mod}_{\bullet}$ and $\varphi :\mathcal{G}_{e_p}(P) \to B$ a morphism in $T\textbf{-Mod}_{\bullet}$.
Now we only to construct a morphism $\psi :\mathcal{G}_{e_p}(P) \to A$ in $T\textbf{-Mod}_{\bullet}$ such that $\pi \circ \psi = \varphi$.

Consider the functor $\mathcal{H}$ described in \cite[Remark 2.16]{mb}. Since $\pi$ is an epimorphism in $T\textbf{-Mod}_{\bullet}$, the underlying set map of $\pi$ is surjective, consequently $\mathcal{H}(\pi) : \mathcal{H}(A) \to \mathcal{H}(B)$ is an epimorphism in $T\textbf{-HMod}$.
Moreover, by \cite[Remark 2.16]{mb}, there exists a natural isomorphism $\mathcal{H}(\mathcal{G}_{e_p}(P)) \cong P$. Thus, we obtain a morphism $\mathcal{H}(\varphi) : P \to \mathcal{H}(B)$ in $T\textbf{-HMod}$.

As $P$ is projective in $T\textbf{-HMod}$, there exists a morphism $\eta : P \to \mathcal{H}(A)$ such that the following diagram commutes in $T\textbf{-HMod}$.\\
$$ \xymatrix{ & P \ar[dl]_{\eta} \ar[d]^{\mathcal{H}(\varphi)} & \\
  \mathcal{H}(A) \ar[r]_{\mathcal{H}(\pi)} & \mathcal{H}(B) }$$ \\

Let $a_0 = \eta(e_p) \in \mathcal{H}(A)$ and $0_A$ the zero element of the pointed $T$-module $A$. Define the translation automorphism $\tau_{a_0}^{0_A} : \mathcal{H}(A) \to \mathcal{H}(A)$ by $\tau_{a_0}^{0_A}(x) = [x, a_0, 0_A] = x - a_0 + 0_A,$ which is an isomorphism of heaps of $T$-modules.

Now construct $\psi = \mathcal{G}(\eta) = \tau_{a_0}^{0_A} \circ \eta$. Since $\psi(t\rhd m) = \tau_{a_0}^{0_A} \circ \eta(t\rhd_{e_p} m) = t\rhd_{\tau_{a_0}^{0_A} \circ \eta(e_p)} \tau_{a_0}^{0_A} \circ \eta(m) = t\rhd_{0_A} \psi(m) = t\rhd \psi(m) $, for all $t\in T$, $m\in M$, $\psi$ is a homomorphism in $T\textbf{-Mod}_{\bullet}$. Next we will show that $\pi \circ \psi = \varphi $. As
$$ \mathcal{H}(\pi) \circ \mathcal{H}(\psi) (p) = \mathcal{H}(\pi) \circ \tau_{a_0}^{0_A} \circ \eta(p) = \tau_{\mathcal{H}(\pi)(a_0)}^{0_B} \circ \mathcal{H}(\pi) \circ \eta(p) = \tau_{\mathcal{H}(\pi)(a_0)}^{0_B} \circ \mathcal{H}(\varphi)(p),$$
for all $p \in \mathcal{G}_{e_p}(P)$ and $\mathcal{H}(\pi)(a_0) = \mathcal{H}(\pi) \circ \eta(e_p) = \mathcal{H}(\varphi)(e_p) = \varphi(e_p) = 0_B$ (which means $\tau_{\mathcal{H}(\pi)(a_0)}^{0_B} = \tau_{0_B}^{0_B} = Id_{\mathcal{H}(B)}$), $\mathcal{H}(\pi) \circ \mathcal{H}(\psi) = \mathcal{H}(\varphi)$, i.e. $\pi \circ \psi = \varphi$. And hence by \cite[Theorem 3.1]{mb}, $\mathcal{G}_{e_p}(P)$ is a projective $R(T)$-module.

Sufficiency: It's direct by applying $\mathcal{H}$.
\end{proof}

\begin{example}
\begin{enumerate}
    \item[(i)] Let $T$ be $T(\mathbb{Z}[\sqrt{-5}])$ and $R(T) \cong G(T,0) \times \mathbb{Z}$, where $0$ is the zero element of the ring $\mathbb{Z}[\sqrt{-5}]$ (in fact, if we choose $e\neq 0$, we have $G(T,e) \times \mathbb{Z} \cong G(T,0) \times \mathbb{Z}$), then $(2,1+\sqrt{-5}) \times \mathbb{Z}$ is a projective $R(T)$-module, and so a projective heap of $T$-modules by Theorem 3.4;
    \item[(ii)] Let $k$ be a field, $T = T(M_n(k))$ and $R(T) \cong G(T,0) \times \mathbb{Z}$, where $0$ is the zero element of the ring $M_n(k)$ (in fact, if we choose $e\neq 0$, we have $G(T,e) \times \mathbb{Z} \cong G(T,0) \times \mathbb{Z}$), then $k^n \times \mathbb{Z}$ is a projective $R(T)$-module, and so a projective heap of $T$-modules by Theorem 3.4.
\end{enumerate}

\end{example}
The following lemma implies that an $R(T)_u$-module can be naturally regarded as an $R(T)$-module and a free $R(T)_u$-module is a free $R(T)$-module. The converse holds.
\begin{lemma}
The category of unital modules over $R(T)_u$ is isomorphic to the category of $R(T)$-modules. A set is a free $R(T)_u$-module if and only if it's a free $R(T)$-module.
\end{lemma}
\begin{proof}
{It follows from \cite[Proposition 2.11]{mb} that we have an isomorphism of unital rings $R(T)_u \cong R(T_u)$. By \cite[Theorem 3.1]{mb}, there exists an isomorphism of categories between the category of pointed $T_u$-modules $T_u\textbf{-Mod}_{\bullet}$ and the category of $R(T_u)$-modules $R(T_u)\textbf{-Mod}$. Combining these we obtain a category isomorphism:
$$T_u\textbf{-Mod}_{\bullet} \cong R(T_u)\textbf{-Mod} \cong R(T)_u\textbf{-Mod}.$$

Now construct an isomorphism of categories between $R(T)\textbf{-Mod}$ and $R(T)_u\textbf{-Mod}$.
Define a functor $F: R(T)\textbf{-Mod} \rightarrow R(T)_u\textbf{-Mod}$ as follows. For an $R(T)$-module $(M, \cdot)$, we endow the same abelian group $M$ with an $R(T)_u$-module structure by
$$ (r, n) \ast m := r \cdot m + n \cdot m, \forall (r, n) \in R(T)_u = R(T) \oplus \mathbb{Z}, m \in M, $$
where $n \cdot m$ denotes the $n$-fold sum of $m$. This is well-defined. It suffices to show that $((r, n)(s, m)) \ast x = (r, n) \ast ((s, m) \ast x) $. Note that for $(r, n), (s, m) \in R(T)_u$ and $x \in M$,
\begin{align*}
((r, n)(s, m)) \ast x &= (rs + ns + mr, nm) \ast x \\
&= (rs + ns + mr) \cdot x + (nm) \cdot x \\
&= r \cdot (s \cdot x) + n \cdot (s \cdot x) + m \cdot (r \cdot x) + (nm) \cdot x
\end{align*}
and
\begin{align*}
(r, n) \ast ((s, m) \ast x) &= (r, n) \ast (s \cdot x + m \cdot x) \\
&= r \cdot (s \cdot x) + r \cdot (m \cdot x) + n \cdot (s \cdot x) + n \cdot (m \cdot x) \\
&= r \cdot (s \cdot x) + m \cdot (r \cdot x) + n \cdot (s \cdot x) + (nm) \cdot x,
\end{align*}
which are equal because $r \cdot (m \cdot x) = m \cdot (r \cdot x)$. The unit $(0,1)$ acts as the identity: $(0,1) \ast x = 0 \cdot x + 1 \cdot x = x$.

For a morphism $f: M \to N$ of $R(T)$-modules, set $F(f) = f$. Since $f$ is $\mathbb{Z}$-linear and commutes with the $R(T)$-action, it also commutes with the $\ast$-action, and hence $F(f)$ is a morphism of $R(T)_u$-modules.

Conversely, consider a ring homomorphism $j: R(T) \rightarrow R(T)_u$ with $r\mapsto (r,0)$. For any $R(T)_u$-module $M$, we define an $R(T)$-module structure on the same abelian group:
$$ r \cdot m := j(r) \cdot_{R(T)_u} m, \forall r \in R(T), m \in M,$$
where $\cdot_{R(T)_u}$ denotes the original $R(T)_u$-action. Now we define a functor $G: R(T)_u\textbf{-Mod} \rightarrow R(T)\textbf{-Mod}$, for an $R(T)_u$-module $(N, \ast)$, $G(N)$ is $N$ with the $R(T)$-action $r \cdot n = (r, 0) \ast n$.

Now we claim that $F$ and $G$ are inverse to each other. For any $R(T)$-module $M$, $G(F(M))$ is $M$ with action $r \cdot' m = (r,0) \ast m = r \cdot m + 0 \cdot m = r \cdot m$, so $G \circ F = \mathrm{Id}_{R(T)\textbf{-Mod}}$. For an $R(T)_u$-module $N$, $F(G(N))$ is $N$ with action $(r, n) \ast' n' = r \cdot n' + n \cdot n'$, where $r \cdot n' = (r,0) \ast n'$. Since $(0,1)$ acts as the identity, we have $(r, n) \ast n' = (r,0) \ast n' + (0, n) \ast n' = r \cdot n' + n \cdot ((0,1) \ast n') = r \cdot n' + n \cdot n'$. Hence $\ast' = \ast$, so $F \circ G = \mathrm{Id}_{R(T_u)\textbf{-Mod}}$.

Therefore, $F$ and $G$ are isomorphic, and the category of unital modules over $R(T)_u$ is isomorphic to the category of $R(T)$-modules.
From the proof above, the isomorphism between $R(T)_u\textbf{-Mod}$ and $R(T)\textbf{-Mod}$ is compatible with the forgetful functors to $\textbf{Set}$, which means a free $R(T)_u$-module is a free $R(T)$-module, and the converse holds.}
\end{proof}
\begin{theorem}\label{th2}
A heap of $T$-modules is projective if and only if it is a coproduct summand of a free heap of $T$-modules.
\end{theorem}
\begin{proof}
{Necessity: Assume that $P$ is a projective heap of $T$-modules. By Corollary \ref{zzz}, we obtain an equivalent functor $\Phi : T\textbf{-HMod} \rightarrow T \textbf{-Mod}_{\bullet}\twoheaddownarrow \mathrm{R}(T)_{u}$. From \cite[Theorem 4.14]{mb}, $\Phi(P) = (\mathcal{G}_{\star}(P \sqcup \ast),\pi_P) \cong (\mathcal{G}_{e_p}(P) \oplus \mathrm{R}(T)_{u},\pi'_P)$, where $\pi_P:\mathcal{G}_{\star}(P \sqcup \ast)\rightarrow R(T)_u$, $\pi'_P:\mathcal{G}_{e_p}(P) \oplus \mathrm{R}(T)_{u}\rightarrow R(T)_u$, $\star$ and $\ast $ are singleton sets. Define an equivalent functor $\Phi' : T\textbf{-HMod} \rightarrow T \textbf{-Mod}_{\bullet}\twoheaddownarrow \mathrm{R}(T)_{u}$ with $\Phi'(P) = (\mathcal{G}_{e_p}(P) \oplus \mathrm{R}(T)_{u},\pi'_P)$. Denote its quasi-inverse functor by $\Psi'$. Since $R(T)\textbf{-Mod}$ is an abelian category, the projective module $\mathcal{G}_{e_p}(P)$ must be a direct summand of a series of free $R(T)$-module $\{F_{RT,i}\}_{i \in I}$, which means that there exists an $R(T)$-module $Q$ such that $ \mathcal{G}_{e_p}(P) \oplus Q \cong \oplus_{i\in I}F_{RT,i}$. For each free $R(T)$-module $F_{RT,i}$, there exists a corresponding free heap of $T$-modules $F_i = \mathcal{H}(F_{RT,i}) \cong \Psi'(F_{RT,i} \oplus R(T)_u, \pi_{i})$, where $\pi_{i}: F_{RT,i} \oplus R(T)_u\rightarrow R(T)_u$. Since left adjoint functors preserve colimits, equivalent functors $\Phi'$ and $\Psi'$ preserve coproduct,
$$\Phi'(\bigsqcup_{i \in I}F_i) = \bigoplus_{i \in I}(F_{RT,i} \oplus R(T)_u, \pi_{i})
 \cong (\bigoplus_{i \in I}(F_{RT,i} \oplus R(T)_u), \sum_{i \in I}\pi_{i}),$$
where
\begin{align*}
\bigoplus_{i \in I}(F_{RT,i} \oplus R(T)_u) &\cong (\mathcal{G}_{e_p}(P) \oplus Q) \oplus (\bigoplus_{i \in I}R(T)_u)\\
 &= (\mathcal{G}_{e_p}(P) \oplus R(T)_u) \oplus (Q \oplus R(T)_u^{I/i_0}).
\end{align*}
So
\begin{align*}\Phi'(\bigsqcup_{i \in I}F_i) &\cong (\mathcal{G}_{e_p}(P) \oplus R(T)_u, \pi_{\mathcal{G}_{e_p}(P)}) \oplus (Q \oplus R(T)_u^{I/i_0}, \pi_Q)\\
 &= \Phi'(P) \oplus (Q \oplus R(T)_u^{I/i_0}, \pi_Q),\end{align*}
where $\pi_{\mathcal{G}_{e_p}(P)}: \mathcal{G}_{e_p}(P) \oplus R(T)_u\rightarrow R(T)_u$, $\pi_Q: Q \oplus R(T)_u^{I/i_0}\rightarrow R(T)_u$. This implies that $\Phi'(P)$ is a direct summand of $\Phi'(\bigsqcup_{i \in I}F_i)$, i.e. $P \cong \Psi'(\Phi'(P))$ is a coproduct summand of $ \Psi'(\Phi'(\bigsqcup_{i \in I}F_i)) \cong \bigsqcup_{i \in I}F_i$.

Sufficiency: Assume that a heap of $T$-module $P$ is a coproduct summand of a series of free heaps of $T$-modules $\{F_i\}_{i \in I}$, i.e. there exists $P'$ such that $\bigsqcup_{i \in I}F_i \cong P \sqcup P'$, then we have
$$\Phi'(\bigsqcup_{i \in I}F_i) \cong \Phi'(P) \oplus \Phi'(P').$$
For each free heap of $T$-modules $F_i$, $\Psi'(F_i) = (\mathcal{G}_{e_i}(F_i) \oplus R(T)_u,\pi_i) = (F_{RT,i} \oplus R(T)_u,\pi_i)$, where $F_{RT,i}$ is a free $R(T)$-module by \cite[Proposition 3.4]{mb}, so
$$\Psi'(\bigsqcup_{i \in I}F_i) = \bigoplus_{i \in I}(F_{RT,i} \oplus R(T)_u,\pi_i) \cong (\bigoplus_{i \in I}(F_{RT,i} \oplus R(T)_u),\sum_{i \in I}\pi_i).$$
Since $\Phi'(\bigsqcup_{i \in I}F_i) \cong \Phi'(P) \oplus \Phi'(P') \cong (\bigoplus_{i \in I}(F_{RT,i} \oplus R(T)_u),\sum_{i \in I}\pi_i)$, $\mathcal{G}_{e_p}(P)$ is a direct summand of a series of free $R(T)$-modules ($R(T)_u$ is a free $R(T)$-module by Proposition 3.6), and hence $\mathcal{G}_{e_p}(P)$ is a projective $R(T)$-module. So $P$ is a projective heap of $T$-modules by Theorem \ref{th1}.
}
\end{proof}

\begin{proposition}\label{th3}\rm
Let $P$ be a heap of $T$-modules. The following conditions are equivalent.
\begin{enumerate}
    \item[\rm(i)] $P$ is projective;
    \item[\rm(ii)] The functor $\operatorname{Hom}_{TH}(P,-): T\textbf{-HMod} \rightarrow \textbf{Ah}$ preserves epimorphisms;
    \item[\rm(iii)] Every short exact sequence $K \stackrel{\tau}{\hookrightarrow} N \stackrel{\pi}{\twoheadrightarrow} P$ splits, i.e. there exists a section morphism $s: P\rightarrow N$ of $\pi$.
\end{enumerate}
\end{proposition}
\begin{proof}
{$\rm (i)\Rightarrow (ii)$: Assume that $P$ is a projective heap of $T$-modules and let $g: M \rightarrow N$ be an epimorphism. Take any $\phi \in \operatorname{Hom}_{TH}(P,N)$, since $P$ is projective, there exists $\psi \in \operatorname{Hom}_{TH}(P,M)$ such that $g \circ \psi = \phi$, i.e. $\operatorname{Hom}_{TH}(P,g)(\psi)=\phi$, which implies $\operatorname{Hom}_{TH}(P,g)$ is surjective and the functor $\operatorname{Hom}_{TH}(P,-): T-HMod \rightarrow Ah$ preserves epimorphisms.

$\rm(ii)\Rightarrow (iii)$: Consider a short exact sequence $K \stackrel{\tau}{\hookrightarrow} N \stackrel{\pi}{\twoheadrightarrow} P$. Since $\operatorname{Hom}_{TH}(P,-): T-HMod \rightarrow Ah$ preserves epimorphisms, $\operatorname{Hom}_{TH}(P,\pi): \operatorname{Hom}_{TH}(P,N) \rightarrow \operatorname{Hom}_{TH}(P,P)$ is epic, and hence there exists $s \in \operatorname{Hom}_{TH}(P.N)$ such that $\pi \circ s = {\rm Id}_P$.

$\rm(iii)\Rightarrow (i)$: Let $g: M \to N$ be an epimorphism and $\phi :P \to N$ arbitary morphism. Construct the pullback $Q = \{(m,p) \in M \times P~|~g(m) = \phi(p)\}$.\\
$$ \xymatrix{ Q \ar[d]_{\pi_M} \ar[r]^{\pi_P} & P \ar[d]_{\phi} \\
  M \ar[r]_{g} & N }$$ \\
Since $g$ is an epimorphism, there exists $m\in M$ such that $g(m) = \phi(p)$, $(m,p) \in Q$ for all $p \in P$, i.e. $\pi_p$ is epic.

Consider $e \in P$, let $K = \pi_P^{-1}(e)$. Then we have a short exact sequence
$$ K \stackrel{\tau}{\hookrightarrow} Q \stackrel{\pi_P}{\twoheadrightarrow} P. $$
By $\rm(iii)$, there exists $s: P\rightarrow Q$ such that $\pi_P \circ s = Id_p$. Define $\psi = \pi_M \circ s: P \rightarrow M$, then
$$ g \circ \psi = g \circ \pi_M \circ s = \phi \circ \pi_P \circ s = \phi \circ Id_P = \phi.$$
Therefore, $P$ is projective.
}
\end{proof}

\begin{proposition}
The heaps of $T$-modules sequence $\bigsqcup_{i=1}^{n_1}M_i\hookrightarrow \bigsqcup_{i=1}^{n_2}M'_i \twoheadrightarrow \bigsqcup_{i=1}^{n_3}M''_i$ is exact if the $R(T)$-modules sequence $\bigoplus_{i=1}^{n_1}\mathcal{G}_{e_i}(M_i)\hookrightarrow \bigoplus_{i=1}^{n_2}\mathcal{G}_{e'_i}(M'_i) \twoheadrightarrow \bigoplus_{i=1}^{n_3}\mathcal{G}_{e''_i}(M''_i)$ is exact and $n_1 + n_3 = n_2+1$, where $e_i \in M_i, e'_i \in M'_i, e''_i \in M''_i$, $n_1, n_2, n_3\in N$.
\end{proposition}
\begin{proof}
From Construction \ref{xu}, $\bigsqcup_{i=1}^{n}M_i \cong \mathcal{H}((\bigoplus_{i=1}^{n}\mathcal{G}_{e_i}(M_i))\oplus(\bigoplus_{i=1}^{n-1}R(T)_u))$. This consequence is direct.
\end{proof}

\begin{construction}\label{oo}
Consider $\operatorname{Hom}_{Ah}(M,N)$, where $M$, $N$ are abelian heaps. We can naturally endow it with a ternary operation that makes it an abelian heap\\
 $$[-,-,-]:\operatorname{Hom}_{Ah}(M,N)\times \operatorname{Hom}_{Ah}(M,N)\times \operatorname{Hom}_{Ah}(M,N)\rightarrow \operatorname{Hom}_{Ah}(M,N),$$
for all $m \in M,$ $[f_1,f_2,f_3](m)=[f_1(m),f_2(m),f_3(m)]$.

For an abelian heap $N$, let $\operatorname{Tr}(N)$ consist of all translation automorphisms ($\tau^b_a:N \rightarrow N$,   $x \mapsto [x,a,b]_N$ for $a,b\in N$). Since for all $\tau_{a_{1}}^{b_{1}},\tau_{a_{2}}^{b_{2}},\tau_{a_{3}}^{b_{3}} \in \operatorname{Tr}(N),[\tau_{a_{1}}^{b_{1}},\tau_{a_{2}}^{b_{2}},\tau_{a_{3}}^{b_{3}}] = \tau_{[a_{1},a_{2},a_{3}]}^{[b_{1},b_{2},b_{3}]} \in \operatorname{Tr}(N)$, $\operatorname{Tr}(N)$ is a sub-heap of abelian heap $\operatorname{End}_{Ah}(N)$.

Now for an abelian heap $M$, define a relation $\sim_{\operatorname{Tr}(N)}$ on $\operatorname{Hom}_{Ah}(M,N)$ by $f\sim_{\operatorname{Tr}(N)}f' \Leftrightarrow \exists \tau \in \operatorname{Tr}(N)$ such that $f'=\tau f$. Since for all $f,g,h\in \operatorname{Hom}_{Ah}(M,N), f=\tau^a_a f$, $f=\tau^a_b g\Leftrightarrow g=\tau^b_a f$, $h=\tau^d_c g,g=\tau^b_a f\Rightarrow h=\tau^{[b,c,d]}_a f$, $\sim_{\operatorname{Tr}(N)}$ is an equivalent relation. By \cite[Remark 2.1]{mb}, any congruence over an abelian heap is a sub-heap relation. Therefore, we can construct a sub-heap $\operatorname{Hom}^\bullet_{Ah}(M, N):= \operatorname{Hom}_{Ah}(M, N)/\operatorname{Tr}(N)$ of $\operatorname{Hom}_{Ah}(M, N)$.

From the construction, it's obvious that $\operatorname{Hom}^\bullet_{TH}(M, N)\cong \mathcal{H}(\operatorname{Hom}_{R(T)}(\mathcal{G}_{e_{M}}(M), \mathcal{G}_{e_N}(N))$, where $e_M\in M$, $e_N\in N$.
\end{construction}
\section{Gorenstein projective objects in the category $T\textbf{-HMod}$}

The definition of Gorenstein projective objects in the category $T\textbf{-HMod}$ is posed. It is shown that a heap of $T$-modules $M$ is BP Gorenstein projective if and only if $\mathcal{G}_{e_m}(M)$ is a Gorenstein projective $R(T)$-module for all $e_m\in M$. Moreover, we give a functorial description of the BP Gorenstein projective dimension.

\subsection{BasePoint-Compatible sequences}

\begin{definition}
A sequence of heaps of $T$-modules $ \mathcal{L}: \cdots \to M_{i-1} \xrightarrow{f_{i-1}} M_i \xrightarrow{f_i} M_{i+1} \xrightarrow{f_{i+1}} M_{i+2} \to \cdots $ is said to be BasePoint-Compatible if there exist $e_i \in M_i$ such that $f_i(e_i) = e_{i+1} , f_{i-1}(M_{i-1}) = f_i^{-1}(e_{i+1})$
for all $i \geq 2$. In particular, when this sequence of heaps of $T$-modules has only three terms$:M_1\hookrightarrow M_2 \twoheadrightarrow M_3$, the concept of BasePoint-Compatible sequences and the concept of exactness are equivalent.
\end{definition}

\begin{definition}
A sequence of heaps of $T$-modules $ \mathcal{L}: \cdots \to M_{i-1} \xrightarrow{f_{i-1}} M_i \xrightarrow{f_i} M_{i+1} \xrightarrow{f_{i+1}} M_{i+2} \to \cdots$ is said to be a TH-Complex if there exist $e_i \in M_i$ such that $f_i(e_i) = e_{i+1} , f_{i-1}(M_{i-1}) \subseteq f_i^{-1}(e_{i+1})$
for all $i \geq 2$.
\end{definition}

Consider a TH-Complex. If $f_{i-1}(M_{i-1}) = f_i^{-1}(e_{i+1})$, then $H^i(\mathcal{L}):= f_i^{-1}(e_{i+1})/f_{i-1}(M_{i-1}) = \ast \cong \bar{e_i}$. For convenience, we denote it by $e_i$.

\begin{proposition}\label{zs}
Let $ \mathcal{L}: \cdots \to M_{i-1} \xrightarrow{f_{i-1}} M_i \xrightarrow{f_i} M_{i+1} \xrightarrow{f_{i+1}} M_{i+2} \to \cdots$ be a sequence of heaps of $T$-modules.
\begin{enumerate}
    \item[\rm(i)] $\mathcal{G}_{e_\mathcal{L}}(\mathcal{L})$ and $\mathcal{G}_{e_\mathcal{L'}}(\mathcal{L})$ are isomorphic, where $e_\mathcal{L}$ and $e_\mathcal{L'}$ are different base point select of $\mathcal{L}$;
    \item[\rm(ii)] If $\mathcal{L}$ is BasePoint-Compatible, then it's corresponding $R(T)$-module sequence $ \mathcal{G}_{e_\mathcal{L}}(\mathcal{L}): \cdots \to \mathcal{G}_{e_{i-1}}(M_{i-1}) \xrightarrow{f_{i-1}} \mathcal{G}_{e_{i}}(M_{i}) \xrightarrow{f_i} \mathcal{G}_{e_{i+1}}(M_{i+1}) \xrightarrow{f_{i+1}} \mathcal{G}_{e_{i+2}}(M_{i+2}) \to \cdots $ is an exact sequence. Conversely, if $\mathcal{G}$ is an exact $R(T)$-module sequence, then $\mathcal{H(G)}$ is BasePoint-Compatible;
    \item[\rm(iii)] Let $\mathcal{L}$ be a BasePoint-Compatible sequence and $\mathcal{L'}: \cdots \to M'_{i-1} \xrightarrow{f'_{i-1}} M'_i \xrightarrow{f'_i} M'_{i+1} \xrightarrow{f'_{i+1}} M'_{i+2} \to \cdots$ be a sequence of heaps of $T$-modules. If $\mathcal{L}$ and $\mathcal{L'}$ are isomorphic, then $\mathcal{L'}$ is BasePoint-Compatible.
\end{enumerate}
\end{proposition}
\begin{proof}
\begin{enumerate}
    \item[(i)] $\mathcal{G}_{e_\mathcal{L}}(\mathcal{L})$ is in the form of $\mathcal{G}_{e_\mathcal{L}}(\mathcal{L}): \cdots \to \mathcal{G}_{e_{i-1}}(M_{i-1}) \xrightarrow{\tau^{e_{i}}_{f_{i-1}(e_{i-1})}f_{i-1}} \mathcal{G}_{e_{i}}(M_{i}) \xrightarrow{\tau^{e_{i+1}}_{f_{i}(e_{i})}f_i} \mathcal{G}_{e_{i+1}}(M_{i+1}) \xrightarrow{\tau^{e_{i+2}}_{f_{i+1}(e_{i+1})}f_{i+1}} \mathcal{G}_{e_{i+2}}(M_{i+2}) \to \cdots$ and $ \mathcal{G}_{e_\mathcal{L'}}(\mathcal{L}): \cdots \to \mathcal{G}_{e'_{i-1}}(M_{i-1}) \xrightarrow{\tau^{e'_{i}}_{f_{i-1}(e'_{i-1})}f_{i-1}} \mathcal{G}_{e'_{i}}(M_{i}) \xrightarrow{\tau^{e'_{i+1}}_{f_{i}(e'_{i})}f_i} \mathcal{G}_{e'_{i+1}}(M_{i+1}) \xrightarrow{\tau^{e'_{i+2}}_{f_{i+1}(e'_{i+1})}f_{i+1}} \mathcal{G}_{e'_{i+2}}(M_{i+2}) \to \cdots$. Now it suffices to prove the diagram below is commutative:
    $$ \xymatrix{\cdots \ar[r] & \mathcal{G}_{e_{i}}(M_{i}) \ar[d]_{\tau^{e'_i}_{e_i}} \ar[r]^{\tau^{e_{i+1}}_{f_{i}(e_{i})}f_i} & \mathcal{G}_{e_{i+1}}(M_{i+1}) \ar[d]_{\tau^{e'_{i+1}}_{e_{i+1}}} \ar[r] & \cdots\\
  \cdots \ar[r] & \mathcal{G}_{e'_{i}}(M_{i}) \ar[r]^{\tau^{e'_{i+1}}_{f_{i}(e'_{i})}f_i} & \mathcal{G}_{e'_{i+1}}(M_{i+1}) \ar[r] & \cdots.} $$
  Since for any $m\in \mathcal{G}_{e_{i}}(M_{i})$,

  \begin{align*}
\tau^{e'_{i+1}}_{f_{i}(e'_{i})}f_i\tau^{e'_i}_{e_i}(m)&=\tau^{e'_{i+1}}_{f_{i}(e'_{i})}f_i([m,e_i,e'_i])\\
&=[[f_i(m),f_i(e_i),f_i(e'_i)],f_{i}(e'_{i}),e'_{i+1}]\\
&=[f_i(m),f_i(e_i),e'_{i+1}]
\end{align*}
and
  \begin{align*}
\tau^{e'_{i+1}}_{e_{i+1}}\tau^{e_{i+1}}_{f_{i}(e_{i})}f_i(m)&=\tau^{e'_{i+1}}_{e_{i+1}}([f_i(m),e_{i+1},f_{i}(e_{i})])\\
&=[[f_i(m),f_{i}(e_{i}),e_{i+1}],e_{i+1},e'_{i+1}]\\
&=[f_i(m),f_i(e_i),e'_{i+1}],
\end{align*}
$\tau^{e'_{i+1}}_{f_{i}(e'_{i})}f_i\tau^{e'_i}_{e_i}=\tau^{e'_{i+1}}_{e_{i+1}}\tau^{e_{i+1}}_{f_{i}(e_{i})}f_i,$ which means the diagram is commutative.
    \item[(ii)] From the relevant definitions, this conclusion is obvious.
    \item[(iii)] Consider the following commutative diagram :\\
    $$ \xymatrix{\mathcal{L}:\cdots \ar[r] & M_{i-1} \ar[d]_{\phi_{i-1}} \ar[r]^{f_{i-1}} & M_{i} \ar[d]_{\phi_i} \ar[r]^{f_{i}} & M_{i+1} \ar[d]_{\phi_{i+1}} \ar[r] & \cdots\\
  \mathcal{L'}:\cdots \ar[r] & M'_{i-1} \ar[r]^{f'_{i-1}} & M'_{i} \ar[r]^{f'_{i}} & M'_{i+1} \ar[r] & \cdots.} $$
Since $\mathcal{L}$ is BasePoint-Compatible, there exists $e_i \in M_i$ such that $f_i(e_i) = e_{i+1}$, $f_{i-1}(M_{i-1})$ $= f_{i}^{-1}(e_{i+1})$ for all $i \geq 2$. It suffices to prove that $\phi_i(e_i) \in M'_i$ satisfy $f'_i(\phi_i(e_i)) = \phi_{i+1}(e_{i+1})$, $f'_{i-1}(M'_{i-1}) = f_{i}^{'-1}(\phi_{i+1}(e_{i+1}))$ for all $i \geq 2$. As $f'_i(\phi_i(e_i)) = \phi_{i+1}(f_i(e_i)) = \phi_{i+1}(e_{i+1})$, $f'_{i-1}(M'_{i-1}) = f'_{i-1}(\phi_{i-1}(M_{i-1})) = \phi_i \circ f_{i-1}(M_{i-1})$ for $ m_{i} \in M_i$ and $m'_{i} =\phi_i(m_{i}) \in M'_i$, $f'_{i}(m'_{i}) = \phi_{i+1}(e_{i+1}) \Leftrightarrow f'_{i}(\phi_i(m_{i})) = \phi_{i+1}(e_{i+1}) \Leftrightarrow \phi_{i+1}(f_i(m_{i})) = \phi_{i+1}(e_{i+1}) \Leftrightarrow f_i(m_{i}) = e_{i+1}$, which means $m'_{i} \in f_{i}^{'-1}(\phi_{i+1}(e_{i+1})) \Leftrightarrow m_{i} \in f_{i}^{-1}(e_{i+1}) = f_{i-1}(M_{i-1}) \Leftrightarrow m'_{i} \in \phi_i \circ f_{i-1}(M_{i-1} ) = f'_{i-1}(M'_{i-1})$. Thus, $f'_{i-1}(M'_{i-1}) = f_{i}^{'-1}(\phi_{i+1}(e_{i+1}))$.
\end{enumerate}
\end{proof}
It's not difficult to figure out that if $\mathcal{L}$ is a TH-Complex, the $\rm(ii)$ and $\rm(iii)$ hold in Proposition \ref{zs} (just replace BasePoint-Compatible with TH-Complex and exact with complex). Moreover, $\rm(i)$ implies that no matter which base point we choose, $\mathcal{G}_{e_\mathcal{L}}(\mathcal{L})$ and $\mathcal{G}_{e'_\mathcal{L}}(\mathcal{L})$ can be seen in the same class, which guarantees a one-to-one correspondence between $\mathcal{L}$ and $\mathcal{G}_{e_\mathcal{L}}(\mathcal{L})$.

\subsection{Homological dimensions of BP Gorenstein projective heaps of $T$-modules}

In this subsection, we establish some foundations of homological dimensions and employ them to characterize BP Gorenstein projective heaps of $T$-modules in the category $T\textbf{-HMod}$

\begin{definition}
Let $M$ be a heap of $T$-modules. A (BP) projective resolution of $M$ is (a BasePoint-Compatible) an exact sequence of heaps of $T$-modules
$\dots \to P_n \to P_{n-1} \to \dots \to P_0 \twoheadrightarrow M $, where each $P_i$ is a projective heap of $T$-modules.
\end{definition}

\begin{definition}\label{sqe}
Let $M$ be a heap of $T$-modules. The (BP) projective dimension $(BP)pd_{TH}(M)$ is defined as the smallest non-negative integer $n$ for which there exists (a BasePoint-Compatible) an exact sequence of heaps of $T$-modules $P_n \hookrightarrow P_{n-1} \to \dots \to P_0 \twoheadrightarrow M$, where each $P_i$ is a projective heap of $T$-modules. If no such finite sequence exists, we set $(BP)pd_{TH}(M) = \infty$.
\end{definition}

\begin{remark}
Let $T\textbf{-HMod}$ has enough projective objects. It follows easily that every object $M$ in $T\text{-HMod}$ admits a BP projective resolution. From Theorem \ref{th1} and any $R(T)$-module admits a projective resolution in $R(T)\textbf{-Mod}$, this consequence is direct.
\end{remark}

\begin{definition}
If $T\textbf{-HMod}$ has enough projective objects, then we define $\operatorname{Ext}_{TH}^n(M,N)$ as the set of equivalence classes of $n$-step extension sequences which is BasePoint-Compatible
$$ N \hookrightarrow X_1 \rightarrow \cdots \rightarrow X_n \twoheadrightarrow M.$$
For a BP projective resolution $P^\bullet \to M$, we have $\operatorname{Ext}_{TH}^n(M, N) = H^n(\operatorname{Hom}^\bullet_{TH}(P^\bullet, N))$.
\end{definition}
For all $M,N \in T\textbf{-HMod}, \operatorname{Ext}_{TH}^n(M,N) \cong \varepsilon^n(M,N)$ as an abelian heap, where $\varepsilon^n(M,N)$ is an abelian heap of equivalence classes of $n$-extensions (Yoneda extensions) of $M$ and the ternary operation of $\varepsilon^n(M,N)$ is induced by Baer Sum(let $[E], [E'] \in \varepsilon^n(M,N)$, define the Baer sum of $[E]$ and $[E']$ as $[E]+[E']:= [({\rm Id}_N,{\rm Id}_N)(E \oplus E')\binom{{\rm Id}_M}{{\rm Id}_M}]$).

\begin{definition}
For any class $\mathcal{X}$ of heaps of $T$-modules, we define the left orthogonal class by
$$^{\bot}\mathcal{X} = \{M \in T\textbf{-HMod} | \operatorname{Ext}_{TH}^i(M,X) = e_i,\text{ for all }X\in \mathcal{X}\text{ and all
}i>0\},$$
where $e_i \in H^i(\operatorname{Hom}^\bullet_{TH}(P^\bullet, X))$, $P^\bullet$ is a BP projective resolution of $M$.\\

Subsequently, whenever the expression $\operatorname{Ext}_{TH}^i(M,X) = e_i$ appears, we will no longer specify the abelian heap to which \(e_i\) belongs.
\end{definition}

\begin{definition}
For any heap of $T$-modules $M$, we define two types of BasePoint-Compatible resolutions.
\begin{enumerate}
    \item[(i)] A left $\mathcal{Q}$-resolution of $M$ is a BasePoint-Compatible sequence $\mathbf{Q}=\cdots \rightarrow Q_1 \rightarrow Q_0 \twoheadrightarrow M$ with $Q_n\in \mathcal{Q}$ for all $n\geq 0$;
    \item[(ii)] A right $\mathcal{Q}$-resolution of $M$ is a BasePoint-Compatible sequence $\mathbf{Q}=M \hookrightarrow Q^0 \rightarrow Q^1 \rightarrow \cdots$ with $Q^n\in \mathcal{Q}$ for all $n\geq 0$,
\end{enumerate}
 where $\mathcal{Q}$ is a class in $T\textbf{-HMod}$. Now let $\mathbf{Q}$ be any (left or right) BasePoint-Compatible $\mathcal{Q}$-resolution of $M$, we say that $\mathbf{Q}$ is proper (co-proper) if the sequence $\operatorname{Hom}^\bullet_{TH}(X, \mathbf{Q})$ ($\operatorname{Hom}^\bullet_{TH}(\mathbf{Q}, X)$) is BasePoint-Compatible for all $X\in \mathcal{Q}$.
\end{definition}

\begin{definition}
A complete projective resolution of heaps of $T$-modules is an exact sequence of projective heaps of $T$-modules $\mathcal{P}: \cdots \to P_1 \xrightarrow{f_1} P_0 \xrightarrow{f} P^0 \xrightarrow{f^1} P^1 \to \cdots$ such that $\operatorname{Hom}^\bullet_{TH}(\mathcal{P}, Q)$ is exact for every projective heap of $T$-modules $Q$. A heap of $T$-modules $M $ is called Gorenstein projective if there exists a complete projective resolution of heaps of $T$-modules with $M \cong \operatorname{Im}(P_0 \to P^0)$. The class of Gorenstein projective heaps of $T$-modules is denoted by $\mathcal{GP}(TH)$. In particular, a complete projective resolution of heaps of $T$-modules is said to be BasePoint-Compatible if $\mathcal{P}$ is BasePoint-Compatible, and $\operatorname{Hom}^\bullet_{TH}(\mathcal{P}, Q)$ is BasePoint-Compatible for every projective heap of $T$-modules $Q$. A heap of $T$-modules $M$ is called BP Gorenstein projective if there exists a BasePoint-Compatible complete projective resolution of heaps of $T$-modules with $M \cong \operatorname{Im}(P_0 \to P^0)$. The class of BasePoint-Compatible Gorenstein projective heaps of modules is denoted by $\mathcal{BPGP}(TH)$.
\end{definition}

\begin{proposition}
Let $\mathcal{L}$ be any sequence of heaps of $T$-modules. Then for every heap of $T$-modules $Q$, $\operatorname{Hom}_{TH}(\mathcal{L}, Q) \cong \operatorname{Hom}^\bullet_{TH}(\mathcal{L}, Q)\times Q$.
\end{proposition}
\begin{proof}
Let $\mathcal{L}$ be any sequence of heaps of $T$-modules, consider the diagram below:
$$ \xymatrix{\cdots \ar[r] & \operatorname{Hom}_{TH}(M_i, Q) \ar[d]_{\tau_{i}} \ar[r]^{f'_{i-1}} & \operatorname{Hom}_{TH}(M_{i-1}, Q) \ar[d]_{\tau_{i-1}} \ar[r] & \cdots \\
  \cdots \ar[r] & \mathcal{H}(\operatorname{Hom}_{R(T)}(\mathcal{G}_{e_{i}}(M_{i}), \mathcal{G}_{e_Q}(Q)))\times Q \ar[r]^{((\tau^{e_i}_{f_{i-1}(e_{i-1})}f_{i-1})', g_{i-1})} & \mathcal{H}(\operatorname{Hom}_{R(T)}(\mathcal{G}_{e_{i-1}}(M_{i-1}), \mathcal{G}_{e_Q}(Q))\times Q \ar[r] & \cdots ,}$$
where $f'_{i-1} = \operatorname{Hom}_{TH}(f_{i-1},Q)$, $(\tau^{e_i}_{f_{i-1}(e_{i-1})}f_{i-1})'=\operatorname{Hom}_{TH}(\tau^{e_i}_{f_{i-1}(e_{i-1})}f_{i-1},Q)$, $\tau_{i}(\phi_i) = (\alpha_i,q_i)=(\tau_{\phi_i(e_i)}^{e_Q}\phi_i,\phi_i(e_i))$, $g_{i-1}(q_i)=(f'_{i-1}\tau^{q_i}_{e_Q}\alpha)(e_{i-1})$ for $\phi_i \in \operatorname{Hom}_{TH}(M_i, Q)$. Now it is to verify $\tau_i$ is isomorphic. Construct $\tau'_i: \tau'_i(\alpha_i ,q_i)=\tau^{q_i}_{e_Q}\alpha_i$, since $\tau_i\tau'_i(\alpha_i ,q_i)=\tau_i(\tau^{q_i}_{e_Q}\alpha_i)=(\tau^{e_Q}_{\tau^{q_i}_{e_Q}\alpha_i(e_i)}\tau^{q_i}_{e_Q}\alpha_i,\tau^{q_i}_{e_Q}\alpha_i(e_i))=(\alpha_i ,q_i)$ and $\tau'_i\tau_i(\phi_i)=\tau'_i(\tau_{\phi_i(e_i)}^{e_Q}\phi_i,\phi_i(e_i))=\tau^{\phi_i(e_i)}_{e_Q}\tau_{\phi_i(e_i)}^{e_Q}\phi_i=\phi_i$, $\tau_i$ is isomorphic. As for all $\phi_i \in \operatorname{Hom}_{TH}(M_i, Q)$,
\begin{align*}
((\tau^{e_i}_{f_{i-1}(e_{i-1})}f_{i-1})', g_{i-1}) \tau_i(\phi_i)&=((\tau^{e_i}_{f_{i-1}(e_{i-1})}f_{i-1})', g_{i-1})(\tau_{\phi_i(e_i)}^{e_Q}\phi_i,\phi_i(e_i))\\
&=(\tau_{\phi_i(e_i)}^{e_Q}\phi_i\tau^{e_i}_{f_{i-1}(e_{i-1})}f_{i-1},(f'_{i-1}\tau^{\phi_i(e_i)}_{e_Q}\tau_{\phi_i(e_i)}^{e_Q}\phi_i)(e_{i-1}))\\
&=(\tau_{\phi_i(e_i)}^{e_Q}\tau^{\phi_i(e_i)}_{\phi_i f_{i-1}(e_{i-1})}\phi_i f_{i-1},\phi_i f_{i-1}(e_{i-1}))\\
&=(\tau^{e_Q}_{\phi_i f_{i-1}(e_{i-1})}\phi_i f_{i-1},\phi_i f_{i-1}(e_{i-1})).
\end{align*}
and $\tau_{i-1}f'_{i-1}(\phi_i)=\tau_{i-1}(\phi_i f_{i-1})=(\tau^{e_Q}_{\phi_i f_{i-1}(e_{i-1})}\phi_i f_{i-1},\phi_i f_{i-1}(e_{i-1}))$, the diagram above is commutative, which means that $\operatorname{Hom}_{TH}(\mathcal{L}, Q) \cong \operatorname{Hom}^\bullet_{TH}(\mathcal{L}, Q)\times Q$.
\end{proof}

\begin{theorem}\label{nm}
A heap of $T$-modules $M$ is BP Gorenstein projective if and only if $\mathcal{G}_{e_m}(M)$ is a Gorenstein projective $R(T)$-module for all $e_m\in M$.
\end{theorem}
\begin{proof}
Necessity: Assume that $M$ is a BP Gorenstein projective heap of $T$-modules. Then there exists a BasePoint-Compatible complete projective resolution of heaps of $T$-module $$\mathcal{P}: \cdots \to P_1 \xrightarrow{f_1} P_0 \xrightarrow{f_0} P^0 \xrightarrow{f^1} P^1 \to \cdots$$ such that $M \cong \operatorname{Im}(f_0)$. From Theorem \ref{th1} and Proposition \ref{zs}, we have a corresponding $R(T)$-module exact sequence
$$\mathcal{G}_{e_{\mathcal{P}}}(\mathcal{P}):\cdots \to \mathcal{G}_{e_1}(P_1) \xrightarrow{f_1} \mathcal{G}_{e_0}(P_0) \xrightarrow{f_0} \mathcal{G}_{e^0}(P^0) \xrightarrow{f^1} \mathcal{G}_{e^1}(P^1) \to \cdots $$
with each term being a projective $R(T)$-module. By Theorem \ref{th1}, Construction \ref{oo} and Proposition \ref{zs}, $\operatorname{Hom}^\bullet_{R(T)}(\mathcal{G}_{e_{\mathcal{P}}}(\mathcal{P}), \mathcal{G}_{e_Q}(Q))$ is exact for every projective $R(T)$-module $\mathcal{G}_{e_Q}(Q)$. From \cite[Remark 2.13]{mb}, $\mathcal{G}_{e_m}(M) \cong \mathcal{G}_{e_m}(Imf_0) \cong Im(\mathcal{G}(f_0))$, where $\mathcal{G}(f_0)$ is the corresponding $R(T)$-module homomorphism of $f_0$, thus $\mathcal{G}(f_0)$ is a homomorphism in $\mathcal{G}_{e_{\mathcal{P}}}(\mathcal{P})$ and $\mathcal{G}_{e_m}(M)$ is a Gorenstein projective $R(T)$-module.

Sufficiency: It's direct by applying $\mathcal{H}$.
\end{proof}

\begin{example}
\begin{enumerate}
    \item[(i)] Let $k$ be a field, $T= T(k[x]/(x^2))$ and $R(T) \cong G(T,0) \times \mathbb{Z}$, where $0$ is the zero element of the ring $k[x]/(x^2)$ (in fact, if we choose $e\neq 0$, then we have $G(T,e) \times \mathbb{Z} \cong G(T,0) \times \mathbb{Z}$), then $(x) \times \mathbb{Z}$ is a Gorenstein projective $R(T)$-module, and so a BP Gorenstein projective heap of $T$-modules by Theorem \ref{nm};
    \item[(ii)] Let $k$ be a field, Morita ring $
\Lambda = \begin{pmatrix}
k[x] & k \\
0 & k
\end{pmatrix}$ and $R(T) \cong G(T,0) \times \mathbb{Z}$, where $0$ is the zero element of the ring $\Lambda$ (in fact, if we choose $e\neq 0$, then we have $G(T,e) \times \mathbb{Z} \cong G(T,0) \times \mathbb{Z}$), then $
\begin{pmatrix}
k \\
0
\end{pmatrix}
 \times \mathbb{Z}$ is a Gorenstein projective $R(T)$-module, and so a BP Gorenstein projective heap of $T$-modules by Theorem \ref{nm}.
\end{enumerate}
\end{example}

\begin{definition}
Let $M$ be a heap of $T$-modules. A (BP) Gorenstein projective resolution of $M$ is (a BasePoint-Compatible) an exact sequence of heaps of $T$-modules
$$ \cdots \to (BP)G_n \to (BP)G_{n-1} \to \cdots \to (BP)G_0 \to M,$$
where each $(BP)G_i$ is a (BP) Gorenstein projective heap of modules.
\end{definition}

\begin{definition}
Let $M$ be a heap of $T$-modules. The (BP) Gorenstein projective dimension $(BP)Gpd_{TH}(M)$ is defined as the smallest nonnegative integer $n$ such that there exists an exact (BasePoint-Compatible) sequence of heaps of $T$-modules
$$(BP)G_n \hookrightarrow (BP)G_{n-1} \to \cdots \to (BP)G_0 \twoheadrightarrow M,$$
where each $(BP)G_i$ is a (BP) Gorenstein projective heap of modules. If no such finite exact (BasePoint-Compatible) sequence exists, then we write $(BP)~Gpd_{TH}(M) = \infty$.
\end{definition}

\begin{theorem}\label{vv}
Let $M$ be a heap of $T$-modules with finite BP projective dimension. Then we have:
$$ pd_{TH}(M) \leq BPpd_{TH}(M) = pd_{R(T)}(\mathcal{G}_{e_m}(M)).$$
Similarly, let $M$ be a heap of $T$-modules with finite BP Gorenstein projective dimension. Then we have:
$$ Gpd_{TH}(M) \leq BPGpd_{TH}(M) = Gpd_{R(T)}(\mathcal{G}_{e_m}(M)).$$
\end{theorem}
\begin{proof}
Let $M$ be a heap of $T$-modules with finite BP projective dimension. By Theorem \ref{th1} and Proposition \ref{zs}, $P^\bullet$ is a BP projective resolution of $M$ if and only if $\mathcal{G}_{e_{P^\bullet}}(P^\bullet)$ is a projective resolution of $G_{e_m}(M)$, thus $BPpd_{TH}(M) = pd_{R(T)}(\mathcal{G}_{e_m}(M))$. Since a BasePoint-Compatible sequence of heaps of $T$-modules is an exact sequence of heaps of $T$-modules, $pd_{TH}(M) \leq BPpd_{TH}(M)$. Similarly, by Proposition \ref{zs} and Theorem \ref{nm}, $BPGpd_{TH}(M) = Gpd_{R(T)}(\mathcal{G}_{e_m}(M))$ and $Gpd_{TH}(M) \leq BPGpd_{TH}(M)$.
\end{proof}

\begin{theorem}\rm\label{pp}
Let $T\textbf{-HMod}$ have enough projective objects and $M$, $N$ heaps of $T$-modules. Fix any base points $e_M \in M, e_N \in N$, denote by $\mathcal{G}_{e_M}(M)$ and $\mathcal{G}_{e_N}(N)$ the associated pointed $T$-modules (hence $R(T)$-modules). Then for every $i\geq 0$, there exists an isomorphism of abelian heaps
$$\operatorname{Ext}_{TH}^i(M,N)\cong \operatorname{Ext}_{R(T)}^i(\mathcal{G}_{e_M}(M),\mathcal{G}_{e_N}(N)),$$
where the isomorphism is independent of the choice of base points, and the ternary operation of $\operatorname{Ext}_{R(T)}^i(\mathcal{G}_{e_M}(M),\mathcal{G}_{e_N}(N))$ is induced by its Baer Sum.
\end{theorem}
\begin{proof}
Let $\Upsilon: N \hookrightarrow X_1 \rightarrow \cdots \rightarrow X_i \twoheadrightarrow M$ be a BasePoint-Compatible sequence and $[\Upsilon]$ the equivalence class of $\Upsilon$ (equivalence of extensions). Applying $\mathcal{G}$ to each term and each morphism, we obtain a sequence of $R(T)$-modules $\mathcal{G}(\Upsilon): \mathcal{G}_{e_N}(N) \hookrightarrow \mathcal{G}_{e_1}(X_1) \rightarrow \cdots \rightarrow \mathcal{G}_{e_i}(X_i) \twoheadrightarrow \mathcal{G}_{e_M}(M)$. By Proposition \ref{zs}(ii), this sequence is exact in $R(T)$-Mod. Now define $\Phi :\operatorname{Ext}_{TH}^i(M,N) \rightarrow \operatorname{Ext}_{R(T)}^i(\mathcal{G}_{e_M}(M),\mathcal{G}_{e_N}(N))$ with $[\Upsilon] \mapsto \mathcal{G}([\Upsilon])$. Next we claim that $\Phi$ is well-defined and independent of base points. In fact, if two extensions $\Upsilon$ and $\Upsilon'$ are equivalent in $\operatorname{Ext}_{TH}^i(M,N)$($[\Upsilon]=[\Upsilon']$), then there exist a chain map between $\Upsilon$ and $\Upsilon'$. Applying $\mathcal{G}$ gives a chain map between the corresponding $R(T)$-module extensions, hence $\Phi([\Upsilon])=\Phi([\Upsilon'])$. If we choose different base points, then the translation automorphisms induce a chain isomorphism between $\mathcal{G}_{e_N}(N) \hookrightarrow \mathcal{G}_{e_1}(X_1) \rightarrow \cdots \rightarrow \mathcal{G}_{e_i}(X_i) \twoheadrightarrow \mathcal{G}_{e_M}(M)$ and $\mathcal{G}_{e_N}(N) \hookrightarrow \mathcal{G}_{e'_1}(X_1) \rightarrow \cdots \rightarrow \mathcal{G}_{e'_i}(X_i) \twoheadrightarrow \mathcal{G}_{e_M}(M)$. By Proposition 4.3(i), the chain map between the corresponding $R(T)$-module extensions is isomorphic. Therefore $\Phi$ is well-defined and independent of base points.

Conversely, let $\Gamma : \mathcal{G}_{e_N}(N) \hookrightarrow Y_1 \rightarrow \cdots \rightarrow Y_i \twoheadrightarrow \mathcal{G}_{e_M}(M)$ be an exact sequence of $R(T)$-modules. Apply the functor $\mathcal{H}$, we obtain a sequence of heaps of $T$-modules $\mathcal{H}(\Gamma) : N \hookrightarrow \mathcal{H}(Y_1) \rightarrow \cdots \rightarrow \mathcal{H}(Y_i) \twoheadrightarrow M$. By Proposition \ref{zs}(ii), this sequence is BasePoint-Compatible in $T\textbf{-HMod}$. Now define $\Psi : \operatorname{Ext}_{R(T)}^i(\mathcal{G}_{e_M}(M),\mathcal{G}_{e_N}(N)) \rightarrow \operatorname{Ext}_{TH}^i(M,N)$ with $ [\Gamma] \mapsto \mathcal{H}([\Gamma])$, $\Psi$ is clearly well-defined.

As for any equivalence class $[\Upsilon]$ of a BasePoint-Compatible sequence of heaps of $T$-modules, $\Psi\Phi([\Upsilon])=\mathcal{H}\mathcal{G}([\Upsilon])\cong[\Upsilon]$ and for any equivalence class $[\Gamma]$ of an exact sequence of $R(T)$-modules, $\Phi\Psi([\Gamma])=\mathcal{G}\mathcal{H}([\Gamma])\cong [\Gamma]$, $\operatorname{Ext}_{TH}^i(M,N)\cong \operatorname{Ext}_{R(T)}^i(\mathcal{G}_{e_M}(M),\mathcal{G}_{e_N}(N))$.
\end{proof}
Based on Theorem \ref{nm}, Theorem \ref{pp} and \cite[Lemma 2.17]{rb}, we can directly obtain the following lemma.

\begin{lemma}\label{yy}
Consider a BasePoint-Compatible sequence of heaps of $T$-modules $K_n\hookrightarrow G_{n-1}\rightarrow \cdots \rightarrow G_0 \twoheadrightarrow M$, where $G_0,\ldots,G_{n-1}$ are BP Gorenstein projective heaps of $T$-modules. Then
$$\operatorname{Ext}_{TH}^i(K_n,L)\cong \operatorname{Ext}_{TH}^{i+n}(M,L)$$
for all heaps of $T$-modules $L$ with finite BP projective dimension and all integers $i>0$.
\end{lemma}

\begin{lemma}\label{kk}
A heap of $T$-modules $M$ is a BP Gorenstein projective heap of $T$-modules if and only if $M$ belongs to the left orthogonal class $^{\bot}\mathcal{P}(TH)$ and admits a co-proper right $\mathcal{P}(TH)$-resolution. Furthermore, if $\mathcal{P}$ is a BasePoint-Compatible complete projective resolution, then $\operatorname{Hom}_{TH}(\mathcal{P}, L)$ is BasePoint-Compatible for all heaps of $T$-modules $L$ with finite BP projective dimension. Consequently, if $M$ is BP Gorenstein projective, then $\operatorname{Ext}_{TH}^{i}(M, L)=e_i$ for all $i>0$ and all heaps of $T$-modules $L$ with finite BP projective dimension.
\end{lemma}
\begin{proof}
Following by Proposition \ref{zs}, Theorem \ref{nm}, Theorem \ref{vv}, Lemma \ref{yy} and \cite[Proposition 2.3]{rb}.
\end{proof}

\begin{proposition}
$\mathcal{BPGP}(TH)$ possesses the following three properties.
\begin{enumerate}
    \item[\rm(i)] $\mathcal{P}(TH) \subseteq \mathcal{BPGP}(TH)$;
    \item[\rm(ii)] For any BasePoint-Compatible sequence of heaps of $T$-modules $G'\hookrightarrow G \twoheadrightarrow G''$, if $G''\in \mathcal{BPGP}(TH)$, then $G'\in \mathcal{BPGP}(TH)$ if and only if $G\in \mathcal{BPGP}(TH)$;
    \item[\rm(iii)] $\mathcal{BPGP}(TH)$ is closed under arbitrary direct sums and direct summands.
\end{enumerate}
\end{proposition}
\begin{proof}
From Proposition \ref{zs}, Theorem \ref{nm} and \cite[Theorem 2.5]{rb}, $\rm(i)$, $\rm(ii)$ follow directly. Since $R(T)_u$ is a free $R(T)$-module which is a Gorenstein projective module, by the construction of coproduct in $T\textbf{-HMod}$, $\rm(iii)$ follows.
\end{proof}

\begin{proposition}\label{ddf}
The following two conclusions hold.
\begin{enumerate}
    \item[\rm(i)] Let $M$ be any heap of $T$-modules and consider two BasePoint-Compatible sequences of heaps of $T$-modules
    $$K_n\hookrightarrow G_{n-1}\rightarrow \cdots \rightarrow G_0 \twoheadrightarrow M,$$
    $$\tilde{K_n}\hookrightarrow \tilde{G_{n-1}}\rightarrow \cdots \rightarrow \tilde{G_0} \twoheadrightarrow M,$$
    where $G_0,\ldots,G_{n-1}$ and $\tilde{G_0},\ldots,\tilde{G_{n-1}}$ are BP Gorenstein projective heaps of $T$-modules. Then $K_n$ is BP Gorenstein projective if and only if $\tilde{K_n}$ is BP Gorenstein projective;
    \item[\rm(ii)] Let $G' \hookrightarrow G \twoheadrightarrow M$ be a BasePoint-Compatible sequence of heaps of $T$-modules, where $G$ and $G'$ are BP Gorenstein projective heaps of $T$-modules, and $\operatorname{Ext}_{TH}^1(M,Q)=e_1$ for all projective heaps of $T$-modules $Q$. Then $M$ is a BP Gorenstein projective heap of $T$-modules.
\end{enumerate}
\end{proposition}
\begin{proof}
Following by Proposition \ref{zs}, Theorem \ref{nm} and \cite[Proposition 2.7, Corollary 2.11]{rb}.
\end{proof}

\begin{proposition}\label{ii}
Let $ K \hookrightarrow G \twoheadrightarrow M $ be a BasePoint-Compatible sequence of heaps of $T$-modules, where $G$ is BP Gorenstein projective. If $M$ is BP Gorenstein projective, then so is $K$. Otherwise we get $BPGpd_{TH} K = BPGpd_{TH} M - 1 \geq 0$.
\end{proposition}
\begin{proof}
Following by Proposition \ref{zs}, Theorem \ref{nm}, Theorem \ref{vv} and \cite[Corollary 2.18]{rb}.
\end{proof}
\subsection {A functorial description of the BP Gorenstein projective dimension}
\begin{theorem}\label{th4}
Let $M$ be a heap of $T$-modules with finite BP Gorenstein projective dimension and $n$ an integer. Then the following conditions are equivalent.
\begin{enumerate}
    \item[\rm(i)] $BPGpd_{TH}(M)\leq n $;
    \item[\rm(ii)] ${\rm Ext}_{TH}^{i}(M, L) = e_i$ for all $i > n$, and for all heaps of $T$-modules $L$ with finite $BPpd_{TH}(L)$;
    \item[\rm(iii)] ${\rm Ext}_{TH}^{i}(M, Q) = e_i$ for all $i > n$, and for all projective heaps of $T$-modules $Q$;
    \item[\rm(iv)] For every BasePoint-Compatible sequence of heaps of $T$-modules $K_{n} \hookrightarrow G_{n-1} \to \cdots \to G_{0} \twoheadrightarrow M$, where $G_{0}, \dots, G_{n-1}$ are BP Gorenstein projective, $K_{n}$ is BP Gorenstein projective.
\end{enumerate}
\end{theorem}
\begin{proof}
$\rm(i) \Rightarrow (ii)$: Since $BPGpd_{TH}(M)\leq n $, there exists a BasePoint-Compatible sequence
$$G_n\hookrightarrow G_{n-1} \rightarrow \cdots \rightarrow G_0\twoheadrightarrow M,$$
where $G_0,\ldots,G_{n}$ are BP Gorenstein projective heaps of $T$-modules. By Lemma \ref{yy} and Lemma \ref{kk}, $\operatorname{Ext}_{TH}^i(G_n,L)\cong \operatorname{Ext}_{TH}^{i+n}(M,L)$ for any integers $i>0$ and $\operatorname{Ext}_{TH}^{i}(G_n, L)=e_i$ for all $i>0$. Therefore, $Ext_{TH}^{i}(M, L) = e_i$ for all $i > n$, and for all heaps of $T$-modules $L$ with finite $BPpd_{TH}(L)$.\\
$\rm(ii) \Rightarrow (iii)$: It's obvious.\\
$\rm(iii) \Rightarrow (iv)$: For any BasePoint-Compatible sequence of heaps of $T$-modules
$$K_{n} \hookrightarrow G_{n-1} \to \cdots \to G_{0} \twoheadrightarrow M,$$
where $G_{0}, \dots, G_{n-1}$ are BP Gorenstein projective. From Lemma \ref{yy}, $\operatorname{Ext}_{TH}^i(K_n,Q)\cong \operatorname{Ext}_{TH}^{i+n}(M,Q)$ for any integers $i>0$ and any projective heap of $T$-modules. By (iii), we obtain $\operatorname{Ext}_{TH}^i(K_n,Q)=e_i$ for any integers $i>0$ and for any projective heap of $T$-modules, which means $K_n\in ^\bot\mathcal{P}(TH)$. Since $K_{n} \hookrightarrow G_{n-1} \to \cdots \to G_{0} \twoheadrightarrow M$ is BasePoint-Compatible, there exists $e_i\in G_i$, $e_n\in K_n$, $e_{-1}\in M$ such that $e_{i-1}=f_{i}(e_i)$, $f_i(G_i)= f^{-1}_{i-1}(e_{i-2})$, $f_n(K_n)= f^{-1}_{n-1}(e_{n-2})$, where $f_{i}:G_i\rightarrow G_{i-1}$, $f_{n}:K_n\rightarrow G_{n-1}$, $f_{0}:G_0\rightarrow M$. Now define $K_i=f^{-1}_i(e_{i-1})$, then we have n short BasePoint-Compatible sequences: $K_i\hookrightarrow G_{i-1} \twoheadrightarrow K_{i-1}$, where $K_0=M$. Since $BPGpd(M)<\infty$, we can conclude that $BPGpd(K_n)<\infty$ by Proposition \ref{ii}. Now let $BPGpd(K_n)= m <\infty$, then there exists a (BP) Gorenstein projective resolution at a length of m.
$$ C_m \hookrightarrow C_{m-1} \rightarrow \cdots \rightarrow C_0 \twoheadrightarrow K_n,$$
where $C_{0}, \dots, C_{m}$ are BP Gorenstein projective. So we have m short BasePoint-Compatible sequences $A_i\hookrightarrow C_{i-1} \twoheadrightarrow A_{i-1}$ for $i=1,\ldots,m$, where $A_m=C_m$ and $A_0=K_n$. By Lemma \ref{yy},
$$\operatorname{Ext}_{TH}^1(A_{n-i},Q)\cong \operatorname{Ext}_{TH}^{i}(K_n,Q)=e_i,$$
for all $i=1,\ldots,m$ and all projective heaps of $T$-modules $Q$. Successively apply Proposition \ref{ddf}, we can conclude that $A_{0}, \dots, A_{m}$ are BP Gorenstein projective. In particular, $A_0=K_n$ is BP Gorenstein projective.\\
$\rm(iv) \Rightarrow (i)$: It's obvious.
\end{proof}

\section{Gorenstein Truss}

The definition of (PA) Gorenstein trusses is introduced. It is shown that a unital truss $T$ is a PA Gorenstein truss if and only if $R(T)$ is an Iwanaga-Gorenstein ring.

\begin{definition}{\rm \cite[Definition 3.25]{rtb}}
An ideal of a truss $(T,[- - -],\cdot)$ is a sub-heap $S$ of $(T,[- - -])$ such that $xs\in S$ and $sx\in S$ for all $x\in T$ and $s\in S$.
\end{definition}
In fact, the ideal defined above is a two-sided ideal. In the above definition, if the condition that $xs\in S$ is removed, it becomes the definition of a right ideal; if the condition that $sx\in S$ is removed, it becomes the definition of a left ideal.

\begin{remark}{\rm \cite[Remark 3.16]{rtb}}
Conditions for a sub-heap $P$ of $(T,[- - -],\cdot)$ to be a paragon are: for all $x\in T$ and all $p,p'\in P$,
$$[xp,xp',p']\in P \, ~and~ \, [px,p'x,p']\in P.$$
The first of equations defines a left paragon, while the second one defines a right paragon.
$$[xp',xp,p']\in P \, ~and~ \, [p'x,px,p']\in P.$$
Indeed, since $P$ is a sub-heap, if $[xp,xp',p']\in P$, then
$$ P\ni [p',[xp,xp',p'],p'] = [xp',xp,p'].$$
The converse follows from the equality $[xp,xp',p'] = [p',[xp',xp,p'],p']$. The equivalence of the right paragon identities is proven in a similar way.
\end{remark}

\begin{definition}
Let $(T, [---], \cdot)$ be a unital truss and $I$ a left or right ideal (paragon) of $T$.
If there exists a finite subset $X = \{x_1, x_2, \dots, x_n\} \subseteq T$ (for some $n \in \mathbb{N}^*$) such that
$$I = \bigcap \left\{ J \subseteq T \mid J \text{ is a left or right ideal (paragon) of } T \text{ and } X \subseteq J \right\},$$
then $I$ is called a finitely generated left or right ideal (paragon) of $T$.
The set $X$ is called a finite generating set of $I$, and we denote $I = TX$ ($XT$). If $I$ is both a finitely generated left ideal (paragon) and a finitely generated right ideal (paragon), then $I$ is called a finitely generated ideal (paragon).
\end{definition}

\begin{definition}
A unital truss $T$ is called a left (PA) Noetherian truss if every left ideal (paragon) of $T$ is finitely generated. Dually, a unital truss $T$ is called a right (PA) Noetherian truss if every right ideal (paragon) of $T$ is finitely generated. A unital truss that is both left and right (PA) Noetherian is called a (PA) Noetherian truss.
\end{definition}
Since every ideal of truss $T$ is paragon, it's clear that every PA Noetherian truss is Noetherian truss.

\begin{lemma}\label{fg}
A unital truss $T$ is a left {\rm PA} Noetherian truss if and only if $R(T)$ is a left Noetherian ring.
\end{lemma}
\begin{proof}
By \cite[Remark 2.10]{mb}, If $T$ is unital, then the ring $R(T)$ can be realised as the abelian group $\mathcal{G}_{1_T}(T)\oplus \mathbb{Z}$ with multiplication $(t,m)(s,n)=(ts+(n-1)t+(m-1)s,mn)$, where $(t,m),(s,n)\in \mathcal{G}_{1_T}(T)\oplus \mathbb{Z}$. Moreover, the multiplication of $\mathcal{G}_{1_T}(T)$ is $t\cdot s=ts-t-s$ for all $t,s \in \mathcal{G}_{1_T}(T)$. Therefore, $R(T)$ is the Dorroh extension of $\mathcal{G}_{1_T}(T)$. Consider a sub-ring $G'$ of $\mathcal{G}_{1_T}(T)$, $G'=\mathcal{G}_{1_T}(\mathcal{T}(G'))$, thus $G'$ can be denoted by $\mathcal{G}_{1_T}(G)$ where $G=\mathcal{T}(G')$. Since $G'$ is a sub-ring of $\mathcal{G}_{1_T}(T)$, $G=\mathcal{T}(G')$ is a sub-truss of $T$.

From the construction of the ideal of Dorroh extension of a ring. If $\mathcal{G}_{1_T}(T)$ is non-unital, then the left ideal of $R(T)$ is in the form of $\mathcal{G}_{1_T}(G)\oplus \{ 0 \}$ or $\mathcal{G}_{1_T}(G)\oplus \mathbb{Z}(a,d)$, the elements in $\mathcal{G}_{1_T}(G)\oplus \mathbb{Z}(a,d)$ is in the form of $(g+ka,kd)$ for all $g\in \mathcal{G}_{1_T}(G), k\in \mathbb{Z}$, where $\mathcal{G}_{1_T}(G)$ is the left ideal of $\mathcal{G}_{1_T}(T)$, $a\in \mathcal{G}_{1_T}(G), d\in \mathbb{Z}$. Otherwise, the left ideal of $R(T)$ is in the form of $\mathcal{G}_{1_T}(G)\oplus d\mathbb{Z}$, where $\mathcal{G}_{1_T}(G)$ is the left ideal of $\mathcal{G}_{1_T}(T)$, $d\in \mathbb{Z}$.

Now we claim that for a sub-heap $G$ of $T$, if $1_T \in G$, then $G$ is a left paragon of $T$ $\Longleftrightarrow$ $\mathcal{G}_{1_T}(G)$ is a left ideal of $\mathcal{G}_{1_T}(T)$. If $1_T \notin G$, then $G$ is a left paragon of $T$ $\Longleftrightarrow$ $\mathcal{G}_{1_T}(\tau^{1_T}_g(G))$ is a left ideal of $\mathcal{G}_{1_T}(T)$ for all $g\in G$.

Sufficiency: For a left paragon $G$ of $T$, $[ts,ts',s']\in G$ for any $t\in T,s,s'\in G$, let $s'=1_T$, then $[ts,ts',s']= [ts,t,1_T] = ts-t$ which means $t \cdot s= ts-t-s \in \mathcal{G}_{1_T}(G)$ for all $t \in \mathcal{G}_{1_T}(T),s\in \mathcal{G}_{1_T}(G)$. Therefore, $\mathcal{G}_{1_T}(G)$ is a left ideal of $\mathcal{G}_{1_T}(T)$.

Necessity: For a left ideal $\mathcal{G}_{1_T}(G)$ of $\mathcal{G}_{1_T}(T)$, $t\cdot s \in \mathcal{G}_{1_T}(G)$ for any $t\in T,s\in G$. Since $t\cdot s=ts-t-s$, $ts=t\cdot s+t+s$ and $[ts,ts',s']=t\cdot s+t+s-(t\cdot s'+t+s')+s'=t\cdot s-t\cdot s'+s$ where $t\in T,s,s'\in G$. Since $\mathcal{G}_{1_T}(G)$ is a left ideal of $\mathcal{G}_{1_T}(T)$, $t\cdot s \in \mathcal{G}_{1_T}(G),t\cdot s' \in \mathcal{G}_{1_T}(G)$ and $[ts,ts',s']\in G$. Therefore, $G$ is a left paragon of $T$.

Since $G\cong \tau^{1_T}_g(G)$ for all $g\in G$, the second claim follows directly from the first. Let $G$ be a left paragon of $T$, then it corresponds to a left ideals of $R(T)$, $\mathcal{G}_{1_T}(G)\oplus \{ 0 \}$, $\mathcal{G}_{1_T}(G)\oplus \mathbb{Z}(a,d)$ or $\mathcal{G}_{1_T}(G)\oplus d\mathbb{Z}$, where $a\in G, d\in \mathbb{Z}$. Since $G$ is finitely generated $\Longleftrightarrow$ $\mathcal{G}_{1_T}(G)$ is finitely generated $\Longleftrightarrow$ $\mathcal{G}_{1_T}(G)\oplus \{ 0 \}$, $\mathcal{G}_{1_T}(G)\oplus \mathbb{Z}(a,d)$ and $\mathcal{G}_{1_T}(G)\oplus d\mathbb{Z}$ is finitely generated, $T$ is a left PA Noetherian truss $\Longleftrightarrow$ $R(T)$ is a left Noetherian ring.
\end{proof}

\begin{corollary}\label{ff}
A unital ring $A$ is left Noetherian if and only if the truss $T(A)$ is a left {\rm PA} Noetherian truss.
\end{corollary}
\begin{proof}
Since $\mathbb{Z}$ is left Noetherian, $A$ is a left Noetherian unital ring if and only if $R(T(A)) \cong G(T(A),e) \times \mathbb{Z} \cong G(T(A),0) \times \mathbb{Z} = A \times \mathbb{Z}$ is left Noetherian. By Lemma \ref{fg}, $A$ is a left Noetherian unital ring if and only if the truss $T(A)$ is a left PA Noetherian truss.
\end{proof}

From Definition \ref{proj}, Definition \ref{sqe} and Theorem \ref{vv}, dually we have the following.
\begin{definition}
A heap of $T$-modules $I$ is said to be injective if for any monomorphism $f:M\rightarrow N$ in $T\textbf{-HMod}$ and any morphism $g:M\rightarrow I$, there exists a morphism $h:N\rightarrow I$ such that $h\circ f= g$. The class of injective heaps of $T$-modules is denoted by $\mathcal{I}(TH)$.
\end{definition}

\begin{definition}
Let $M$ be a heap of $T$-modules. The (BP) injective dimension $(BP)id_{TH}(M)$ is defined as the smallest non-negative integer $n$ for which there exists an exact (BasePoint-Compatible) heaps of $T$-modules sequence $M \hookrightarrow E^{0} \to E^{1} \to \dots \twoheadrightarrow E^n $, where each $E^i$ is a injective heap of $T$-modules. If no such finite sequence exists, we set $(BP)id_{TH}(M) = \infty$.
\end{definition}

\begin{theorem}\label{dua}
Let $M$ be a heap of $T$-modules with finite BP injective dimension. Then we have:
$$ id_{TH}(M) \leq BPid_{TH}(M) = id_{R(T)}(G_{e_m}(M)).$$
\end{theorem}

\begin{definition}
A unital truss $T$ is called a (PA) Gorenstein truss if it satisfies the following conditions.
\begin{enumerate}
    \item[(i)] $T$ is a (PA) Noetherian truss;
    \item[(ii)] The injective dimension of $T$ as a left heap of $T$-modules satisfies $BPid_{TH}(_TT) < \infty$;
    \item[(iii)] The injective dimension of $T$ as a right heap of $T$-modules satisfies $BPid_{TH}(T_T) < \infty$.
\end{enumerate}
\end{definition}

\begin{remark}
For a non-unital truss $T$, we define it as a PA Gorenstein truss if its unitalization $T_u$ is a PA Gorenstein truss, which is equivalent to that the ring $R(T)_u$ is an Iwanaga-Gorenstein ring.
\end{remark}

\begin{theorem}\label{rt}
Let $T$ be a unital truss. Then $T$ is a {\rm PA} Gorenstein truss if and only if $R(T)$ is an Iwanaga-Gorenstein ring.
\end{theorem}

\begin{proof}
By Lemma \ref{fg}, a unital truss $T$ is a PA Noetherian truss if and only if $R(T)$ is a Noetherian ring. By Theorem \ref{dua}, $BPid_{TH}(M)= id_{R(T)}(G_{e_m}(M))$, which means $BPid_{TH}(T) < \infty$ if and only if $id_{R(T)}(G_{e_m}(M)) < \infty$. Therefore, $T$ is a Gorenstein truss if and only if $R(T)$ is an Iwanaga-Gorenstein ring.
\end{proof}

\begin{corollary}\label{sj}
$R$ is an Iwanaga-Gorenstein unital ring if and only if $T(R)$ is a {\rm PA} Gorenstein truss.
\end{corollary}
\begin{proof}
By Corollary \ref{ff}, a unital ring $R$ is left Noetherian if and only if the truss $T(R)$ is a left PA Noetherian truss. Since the functor $\mathcal{H}$ preserves the underlying elements of objects and maps exact sequences of $R$-modules to basepoint-compatible sequences of heaps of $T(R)$-modules, $id_{R}(M) = BPid_{T(R)H}(\mathcal{H}(M))$, which means $id_{R}(M) < \infty$ if and only if $BPid_{T(R)H}(\mathcal{H}(M)) < \infty$. Therefore, $R$ is an Iwanaga-Gorenstein unital ring if and only if $T(R)$ is a PA Gorenstein truss.
\end{proof}

\paragraph{Competing Interests}
The authors declare that they have no conflicts of interest to this work.

\paragraph{Data Availability}
Data sharing not applicable to this article as no datasets were generated or analysed during the current study.

\textbf{Yongduo Wang}\\
Department of Applied Mathematics, Lanzhou University of Technology, 730050 Lanzhou, Gansu, P. R. China\\
E-mail: \textsf{ydwang@lut.edu.cn}\\[0.3cm]
\textbf{Chenyu Wang}\\
Department of Applied Mathematics, Lanzhou University of Technology, 730050 Lanzhou, Gansu, P. R. China\\
E-mail: \textsf{3217502669@QQ.com}\\[0.3cm]
\textbf{Jian He}\\
Department of Applied Mathematics, Lanzhou University of Technology, 730050 Lanzhou, Gansu, P. R. China\\
E-mail: \textsf{jianhe30@163.com}\\[0.3cm]
\textbf{Dejun Wu}\\
Department of Applied Mathematics, Lanzhou University of Technology, 730050 Lanzhou, Gansu, P. R. China\\
E-mail: \textsf{wudj@lut.edu.cn}\\[0.3cm]
\end{document}